\documentclass[10pt,reqno]{amsart}
\usepackage{amsmath,amssymb,amsfonts,amsthm}
\usepackage[left=3cm, right=3cm, top=2.8cm, bottom=2.8cm]{geometry}
\usepackage[utf8]{inputenc}
\usepackage[T1]{fontenc}
\usepackage{xcolor}
\usepackage{graphicx}
\usepackage{hyperref}
\hypersetup{
    colorlinks=true,
    linkcolor=blue,
    filecolor=magenta,
    urlcolor=cyan,
}
\newtheorem{theorem}{Theorem}
\newtheorem{lemma}[theorem]{Lemma}
\newtheorem{proposition}[theorem]{Proposition}
\newtheorem{corollary}[theorem]{Corollary}
\newtheorem{definition}[theorem]{Definition}
\newtheorem{remark}[theorem]{Remark}
\newtheorem{example}[theorem]{Example}

\theoremstyle{plain}

\begin{document}

\title[Space-Dependent Fractional Evolution Equations]{Space-Dependent Fractional Evolution Equations: A New Approach}

%%%%%%%%%%%%%%%%%%%%%%%%%%%

\author[T. A. S. Boza]{Tiago A. S. Boza}
\address[Tiago A. S. Boza]{Department of Mathematics, Federal University of Santa Catarina, Florian\'{o}polis - SC, Brazil\vspace*{0.3cm}}
\email[]{boza.tiago@ufsc.br}
\author[P. M. Carvalho-Neto]{Paulo M. Carvalho-Neto}
\address[Paulo M. de Carvalho Neto]{Department of Mathematics, Federal University of Santa Catarina, Florian\'{o}polis - SC, Brazil\vspace*{0.3cm}}
\email[]{paulo.carvalho@ufsc.br}

%%%%%%%%%%%%%%%%%%%%%%%%%%%


%%%%%%%%%%%%%%%%%%%%%%%%%%%


%%%%%%%%%%%%%%%%%%%%%%%%%%%

\begin{abstract}
%%%
Inspired by the works of \cite{baz2} and \cite{kian}, this study develops an abstract framework for analyzing differential equations with space-dependent fractional time derivatives and bounded operators. Within this framework, we establish existence and uniqueness results for solutions in both linear and semilinear settings. Our findings provide deeper insights into how spatially varying fractional derivatives influence the behavior of differential equations, shedding light on their mathematical properties and potential applications.
%%%
\end{abstract}

\maketitle

%%%%%%%%%%%%%%%%%%%%%%%%%%%

\section{Introduction}

The research on ordinary and partial differential equations with time fractional derivatives has attracted significant attention over the past few decades. An excellent reference on some of the various contributions that fractional derivatives make to science in general, and particularly in the fields of mathematical physics, engineering, and biology, can be found in \cite{car}. In this article, Carcione et al. investigate several applications of mathematical formalisms, highlighting, for example, the physical relevance that justifies the study of fractional diffusion equations. According to Carcione et al., these equations were initially introduced by Hilfer, Anton, and Compte in the context of a continuous-time random walk model \cite{compte,hilfer}.

Recently, the study of kinetic equations with fractional time derivatives has become an important tool for describing anomalous diffusion and relaxation phenomena. According to Metzler and Klafter in \cite{metzler}, such fractional kinetic equations for diffusion, diffusion-advection, and Fokker-Planck-like processes are presented as a useful approach to describing the transport dynamics in complex systems governed by anomalous diffusions and non-exponential relaxation patterns.

Within the scope of these studies, \cite{metzler} offers an intuitive differentiation between classical and anomalous diffusion processes. For example, in classical diffusion processes like the heat equation, the mean square displacement of a particle is proportional to $t$, while in an anomalous diffusion process like the fractional heat equation, it is proportional to $t^{\alpha}$. Furthermore, the authors provide a physical justification for the need to deduce a "better" differential equation capable of describing certain anomalous behaviors. In their studies, they conclude that the fractional derivative (Riemann-Liouville or Caputo) would be a more suitable choice.

We can, therefore, inquire about the advantages of adopting an approach based on fractional differential equations for diffusion processes. In this regard, the authors in \cite{metzler} summarizes these advantages in the following points:

\begin{itemize}
	\item[(i)] they respond more appropriately to the typical anomalous characteristics that are experimentally observed in many systems;
	\item[(ii)] we can directly include external force terms in the equation;
	\item[(iii)] standard techniques for solving partial differential equations or calculating related transport moments also apply to fractional equations;
	\item[(iv)] the appearance of fractional equations is very appealing due to their proximity to the usual first-order equations;
	\item[(v)] it is not just a new way to present ideas but rather a framework that proves useful for interpreting many other problems and situations.
\end{itemize}

Autonomous Cauchy problems involving fractional diffusion equations of constant order \eqref{equation diffusion constant}, that is,
\begin{equation}\tag{CO} \label{equation diffusion constant}
D_t^{\alpha} u(t, x) = K\left[\dfrac{\partial^2 u(t, x)}{\partial x^2}\right], \quad 0 <\alpha< 1,
\end{equation}
have been extensively studied in the mathematical literature, as mentioned, for example, by \cite{kian, sun}, and other classical references cited by them. Furthermore, recent research, including studies by \cite{goulart, liu, makris}, has once more confirmed the accuracy of using fractional differential equations like \eqref{equation diffusion constant} to analyze experimental data, demonstrating their effectiveness in capturing anomalous diffusion effects. These findings underscore the significance of this theory and provide valuable insights into its application in enhancing our understanding of natural phenomena.

On the other hand, it is necessary to conduct additional theoretical and numerical investigations to incorporate suitable tools for describing more realistic diffusion processes. For example, it is essential to consider stochastic diffusion, as discussed in \cite{chechkin}, since there are various physical, biological, and physiological diffusion phenomena where equations of the type \eqref{equation diffusion constant} are not capable of adequately characterizing, as mentioned by \cite{sun}.

Moreover, according to \cite{fedotov}, fractional dynamic models \eqref{equation diffusion constant} are not sufficiently robust to deal with such problems, especially in heterogeneous media, where permeability may vary at different spatial positions. In such cases, the variable-order model \eqref{equation diffusion space dependentt}, that is,
\begin{equation}\tag{VO} \label{equation diffusion space dependentt}
D_t^{\alpha (t,x,u(t,x))} u(t, x) = K\left[\dfrac{\partial^2 u(t, x)}{\partial x^2}\right], \quad 0 <\alpha (t,x,y) < 1,
\end{equation}
seems to be the most suitable approach to correctly describe diffusion processes. In fact, in \cite{west}, it is observed that some phenomena exhibit complex characteristics in terms of analysis, and diffusion behaviors may depend on the evolution of time, spatial variation, or, in more complex cases, even on system parameters.

However, although these considerations emphasize the importance of exploring more comprehensive and flexible approaches to studying diffusion processes, taking into account the specific characteristics of each problem and the nature of the media in which they occur significantly increases the difficulties. As pointed out by \cite{kian}, such studies are generally not available in the literature.

For this reason, our final goal is to develop a robust theory for problems of the type \eqref{equation diffusion space dependentt}, focusing on the specific subcase of space-dependent equations \eqref{equation diffusion concentration dependent1}, that is,
\begin{equation}\tag{SD} \label{equation diffusion concentration dependent1}
D_t^{\alpha(x)} u(t, x) = A u(t, x), \quad 0 < \alpha(x) < 1,
\end{equation}
where $A:D(A)\subset X\rightarrow X$ is a linear operator on a Banach space $X$, which are equations that consider that the variation in the order of differentiation depends only on the position $x$ considered.

Throughout our studies, we encountered numerous technical difficulties that rendered the theory both delicate and extensive, potentially making the paper excessively long and difficult to follow. With this in mind, we chose to first address the case with bounded operators, which, while seemingly more manageable, still presented challenges. Additionally, this work serves as a foundation for developing inequalities and tools that will be useful in handling the case of unbounded operators, to be explored in a future paper. 

Thus, the aim of this work is to present a rigorous and abstract framework that establishes the conditions for existence and uniqueness of solutions to abstract Cauchy problems with bounded operators in an appropriate Banach space $X$, as given by \eqref{equation diffusion concentration dependent1}, in both linear and semilinear cases.

\subsection{Motivation and Challenges in Solving Space-Dependent Fractional Differential Equations} 

Here, we aim to highlight the complexities that may arise during the study of differential equations such as \eqref{equation diffusion concentration dependent1}. Let us consider the Cauchy problem
\begin{equation} \label{ode} 
\begin{array}{ll}
cD_t^{\alpha(x)} u(t, x) = \lambda(x)u(t, x), \quad \forall t \geq 0\textrm{ and }\forall x \in [0, L], \\
u(0, x) = u_0(x), \quad \forall x \in [0, L], 
\end{array}
\end{equation}
where $L > 0$, $\alpha \in B\big([0, L]\big)$ (the space of bounded real functions on $[0, L]$) such that $\alpha(x) \in (0, 1)$ for every $x \in [0, L]$, $\lambda, u_0 \in C\big([0, L]\big)$ (the space of continuous real functions on $[0, L]$), and $cD_t^{\alpha(x)}$ denotes the Caputo fractional derivative of order $\alpha(x)$ (for more details see Section \ref{fracintro}).  Then, notice that for every \( x \in [0, L] \) fixed, from the theory of fractional calculus and differential equations, we can deduce that a solution of \eqref{ode} is given by
\begin{equation*}u(t,x) = E_{\alpha(x)}\big(t^{\alpha(x)}\lambda(x)\big)u_0(x).\end{equation*}

It would be expected that by adapting known methods of solving partial differential equations, we should obtain an analogous approach to find the solution of partial differential equations with space-dependent fractional derivative, but this is not the case.

For instance, let us try to adapt the standard method of separation of variables to solve the fractional diffusion space-dependent version of the 1D heat equation. To that end, consider \( L > 0 \) and assume that \( \alpha \in B\big([0, L]\big) \) is such that $\alpha(x)\in(0,1)$ for every $t\in[0,L]$, \( u_0 \in C\big([0, L]\big) \), and that \( cD_t^{\alpha(x)} \) denotes the Caputo fractional derivative of order \(\alpha(x)\). With this, we may consider 
\begin{equation} \label{1dheatequation} 
\begin{array}{ll}
cD_t^{\alpha(x)} u(t, x) = u_{xx}(t, x), \quad \forall t \geq 0\textrm{ and } \forall x \in [0, L], \\
u(0, x) = u_0(x), \quad \forall x \in [0, L], \\
u(t, 0) = u(t, L) = 0, \quad \forall t \geq 0.
\end{array}
\end{equation}

If we try to find all nontrivial solutions of the partial differential equation \eqref{1dheatequation} in the product form \( u(t,x) = f(t)g(x) \), by substituting this form into the equations, we obtain that
\begin{equation}\label{sepvar01} 
\dfrac{cD_t^{\alpha(x)}f(t)}{f(t)} = \dfrac{g^{''}(x)}{g(x)}.
\end{equation}

Note that in the classical situation, the expression above would be equal to a constant, as each side would depend on distinct variables. However, in this case, both sides of the equation continue to depend on $x$. Hence, if we assume that they are equivalent to an $x$-dependent function, such as \(\lambda(x)\), we would obtain from \eqref{1dheatequation} and \eqref{sepvar01} the following two differential equations:
\begin{subequations}
  \begin{align}
  \label{eq1.1}
  &cD_t^{\alpha(x)} f(t) = \lambda(x)f(t), \quad \forall t \geq 0\textrm{ and } \forall x \in [0, L], \\
  \label{eq2.1}
  &g^{''}(x) - \lambda(x) g(x) = 0,\quad \forall x \in [0, L].
\end{align}
\end{subequations}

Due to the boundary condition $u(t, 0) = u(t, L) = 0$, which leads to the equalities $f(t)g(0) = 0$ and $f(t)g(L) = 0$, and because we are interested in nontrivial solutions, it follows that $g(0) = g(L) = 0$. Therefore, from \eqref{eq2.1} we conclude that we are seeking nontrivial solutions to the second-order boundary-value problem
\begin{equation}\label{ode2dim}
\begin{array}{ll}
g^{''}(x) - \lambda(x) g(x) = 0,\quad \forall x \in [0, L], \\
g(0)= g(L)=0.
\end{array}
\end{equation}

On the other hand, even if we consider that the variable $x$ is fixed in equation \eqref{eq1.1}, it is not straightforward to conclude that, as in problem \eqref{ode}, the following holds:
\begin{equation} \label{aux02}f(t) = E_{\alpha(x)}\big(t^{\alpha(x)}\lambda(x)\big).\end{equation}

In fact, the general forms of \(\alpha(x)\) and \(\lambda(x)\) are too broad to ensure that the right-hand side of \eqref{aux02} is independent of $x$, as this independence is necessary for the separation of variables in \eqref{sepvar01}; otherwise, the separation of variables of the equation would be compromised.

To address this problem in our setup, one approach is to appropriately specify \(\alpha(x)\) to ensure better control over these equations. That said, let us assume that 
$$\alpha(x)=\left\{\begin{array}{ll}\alpha_1,&\textrm{ for }x\in[0,L/2],\\
\alpha_2,&\textrm{ for }x\in(L/2,L],\end{array}\right.$$
with $\alpha_1,\alpha_2\in(0,1)$. 

From \eqref{sepvar01}, we conclude that \(\lambda(x)\) is constant within each of the intervals \([0, L/2]\) and \((L/2, L]\). That is, we choose to maintain the same constant value within each interval. Consequently, from \eqref{ode2dim}, with $\lambda_n(x) = -(n\pi)^2/L^2$, we deduce that for each $ n \in \mathbb{N}$, we have
$$g_n(x) = \sin\left(\dfrac{n\pi x}{L}\right),$$
are the desired solutions, identical to the classical case. Now, since from \eqref{aux02} we deduce
$$f(t) = E_{\alpha(x)}\left(\dfrac{-n^2\pi^2 t^{\alpha(x)}}{L^2}\right) = \begin{cases}
E_{\alpha_1}\left(\dfrac{-n^2\pi^2 t^{\alpha_1}}{L^2}\right), & \text{for } x \in [0, L/2], \vspace*{0.2cm}\\
E_{\alpha_2}\left(\dfrac{-n^2\pi^2 t^{\alpha_2}}{L^2}\right), & \text{for } x \in (L/2, L],
\end{cases}$$
 if we can express \(u_0(x)\) as the sine series
$$u_0(x) = \sum_{j=1}^\infty a_j \sin\left(\dfrac{n\pi x}{L}\right),$$
then we obtain
$$u(t,x) = \sum_{j=1}^\infty a_j \sin\left(\dfrac{j\pi x}{L}\right) E_{\alpha(x)}\left(\dfrac{-j^2\pi^2 t^{\alpha(x)}}{L^2}\right),$$
which is a classical solution in $[0,L/2)\cup(L/2,L]$, or in other words, satisfies the first equation of \eqref{1dheatequation} for every $t\geq0$ and almost every \(x \in [0, L]\).

From this construction, two initial conclusions can be drawn. 
\begin{itemize}
\item[(i)] Depending on the choice of \(\alpha(x)\), it is unlikely that the equation can be satisfied at every point in \([0, L]\). Therefore, it is more natural to choose $\alpha \in L^\infty(0, L)$, instead of $\alpha \in B([0, L])$. \vspace*{0.2cm}
\item[(ii)] For non constant \(\alpha(x)\), it is very difficult to choose \(\lambda(x)\) in \eqref{sepvar01} to ensure that \eqref{aux02} is independent of the \(x\) variable. Hence, a general theory is needed to handle the solution of these equations, even if an explicit solution is not trivially possible to be obtained.
\end{itemize}

%%%%%%%%%%%%%%%%%%%%%%%%%%%%%%%%%%%%%%%%%%%%%%%%%%%%%%%%%%%%%%%%%

\subsection{Summary of the Manuscript}

In Section \ref{fracintro}, we introduce classical function spaces and define the Riemann-Liouville fractional integral and derivative of order $\alpha(x)$, along with the Caputo fractional derivative of order $\alpha(x)$. We also prove an important and intricate result (see Theorem \ref{espfunc02}), which serves as the foundation for the inequalities used throughout the paper. Section \ref{linearprob} addresses the linear problem involving bounded operators, which can be treated more explicitly, while Section \ref{last} focuses on the semilinear problem with bounded operators. Finally, Section \ref{applications} offers examples of problems that this theory addresses and suggests potential future work, including the application of numerical methods to enhance understanding of classical models.

%%%%%%%%%%%%%%%%%%%%%%%%%%%%%%%%%%%%%%%%%%%%%%%%%%%%%%%%%%%%%%%%%%%%

\section{Prerequisites on the Theory of Variable Order}
\label{fracintro}

This section introduces key theoretical concepts that enhance the completeness and conciseness of our work. Specifically, we focus on introducing the Riemann-Liouville fractional integral and the Caputo fractional derivative of variable order, along with their respective properties, which are employed throughout this manuscript. In this work we assume that $I$ represents the interval $[0,T]$ or $[0,\infty)$, $\Omega$ is a subset of $\mathbb{R}^n$ for some $n\in\mathbb{N}$, and $\alpha:\Omega\rightarrow[0,1]$. For the remainder of this chapter, we assume that $X$ is an arbitrary Banach space.

Before delving into the definitions of integration and differentiation in the context of variable order, we recall the notations of some classical spaces.

\begin{definition} We refer to the Bochner-Lebesgue spaces $(L^p ({I}; X), \| \cdot \|_{L^p({I}; X)})$, for $1 \leq p \leq \infty$, as the set consisting of all Bochner-measurable functions from ${I} \subset \mathbb{R}$ to $X$ such that the function $t\mapsto\| u (t) \|_{X}$ belongs to $L^p({I})$. Furthermore, its norm is given by:
\begin{equation*}
\| u \|_{L^p({I}; X)} := \left( \displaystyle \int_{{I}} \| u(t) \|_{X}^{p} dt \right)^{\frac{1}{p}}.
\end{equation*}
In this case, $(L^p ({I}; X), \| \cdot \|_{L^p({I}; X)})$ is a Banach space. 
\end{definition}

\begin{definition}  We refer to $C({I}; X)$, as the set consisting of all continuous functions from ${I} \subset \mathbb{R}$ to $X$. If ${I}$ is compact, we can define its norm as:
\begin{equation*}
\| u \|_{C({I}; X)} := \displaystyle \sup_{t \in {I}} \| u(t) \|_{X}.
\end{equation*}
In this case, $(C({I}; X), \| \cdot \|_{C({I};X)})$ is a Banach space.
\end{definition}

\begin{definition}
We denote $C^1({I};X)$ as the set of all continuously differentiable functions from ${I} \subset \mathbb{R}$ to $X$. If ${I}$ is a compact interval, its norm can be defined as follows:
\begin{equation*}
	\| u \|_{C^1({I};X)} := \| u \|_{C({I};X)}+\| u^\prime \|_{C({I};X)}.
\end{equation*}
In this case $(C^1({I},;X), \| \cdot \|_{C({I};X)})$  is a Banach space.
\end{definition}

\begin{definition}
We refer to the Bochner-Sobolev function space as $(W^{p, k} ({I};X), | \cdot |_{W^{p, k}({I}; X)})$ for $1 \leq p \leq \infty$ and $k\in\mathbb{N}$. This space comprises all functions in $L^p ({I};X)$ with weak derivatives of order at most $k$ also belonging to $L^p ({I}; X)$. Its norm is defined as
\begin{equation*}
	\| u \|_{W^{p, k} ({I};X)} := \left( \displaystyle \sum_{n = 0}^{k} \int_{{I}} \| D^n u(t) \|_{X}^{p} dt \right)^{\frac{1}{p}}.
\end{equation*}
Here, the symbol $D^n$ represents the $n$th weak derivative of the function. With this notion, $(W^{p,k} ({I};X), \| \cdot\|_{W^{p, k}({I};X)})$ constitutes a Banach space.
\end{definition}

\begin{remark} In the definitions given above, when $X = \mathbb{R}$, we shall omit $X$ in the notations $L^p (I; X) \), \( C(I; X) \), and \( W^{p, k} (I; X) \), and write instead \( L^p (I) \), \( C(I) \), and \( W^{p, k} (I)$, respectively. The same applies to \( L^p(\Omega) \) for \( 1 \leq p \leq \infty \), which always denotes the space of Lebesgue measurable functions from \(\Omega\) to \(\mathbb{R}\).
\end{remark}

Now we are ready to introduce the differentiation and integration of variable order.

\begin{definition}
The Riemann-Liouville fractional integral of order $\alpha(x)$ of $\varphi:I\times\Omega\rightarrow\mathbb{R}$ is given by
$$J_t^{\alpha(x)}\varphi(t,x):=\left\{\begin{array}{ll}
\dfrac{1}{\Gamma(\alpha(x))}\displaystyle\int_0^t(t-s)^{\alpha(x)-1}\varphi(s,x)ds,&\textrm{if }\alpha(x)\in(0,1],\vspace*{0.2cm}\\
\varphi(t,x),&\textrm{if }\alpha(x)=0.\end{array}\right.$$
for every $(t,x)\in I\times\Omega$ such that the right side of the above identity exists. It worths to emphasize that $\Gamma(\cdot)$ is used here to denote the Gamma function.
\end{definition}

\begin{remark} Recall that if $\varphi\in L^1(0,T;L^1(\Omega))$, then for almost every $x\in\Omega$, the function $\varphi(\cdot,x)$ belongs to $L^1(I)$. Therefore, if $\alpha\in L^1(\Omega)$, then from the classical theory of the Riemann-Liouville fractional integral (cf. \cite[Theorem 2.5]{CaFe1}), we have that $J_t^{\alpha(x)}\varphi(t,x)$ exists for almost every $t\in I$ and almost every $x\in\Omega$.
\end{remark}

To illustrate the definition, consider the functions $\alpha:[0,1]\rightarrow[0,1]$ defined as $\alpha(x)=x$, and $\phi:[0,\infty)\times[0,1]\rightarrow\mathbb{R}$ defined as $\phi(t,x)=t^{1/2}\cos(x)$. Then we have that
$$J_t^{\alpha(x)}\phi(t,x):=\left\{\begin{array}{ll}
\dfrac{\Gamma(3/2)t^{x+(1/2)}\cos(x)}{\Gamma\big(x+(3/2)\big)},&\textrm{if }x\in(0,1]\textrm{ and }t\in[0,\infty),\vspace*{0.2cm}\\
t^{1/2},&\textrm{if }x=0\textrm{ and }t\in[0,\infty).\end{array}\right.$$

Another important point to emphasize is that we can consider $\beta>0$ and define $g_\beta:\mathbb{R}\rightarrow\mathbb{R}$ as
\begin{equation*} g_{\beta}(t)=\left\{\begin{array}{ll}\dfrac{t^{\beta-1}}{\Gamma(\beta)},&\textrm{for }t>0\textrm{ and }\beta>0,\vspace{0.2cm}\\0,&\textrm{for }t\leq0\textrm{ and }\beta>0,\end{array}\right.\end{equation*}
in order to deduce that
$$J_t^{\alpha(x)}\varphi(t,x)=[g_{\alpha(x)}*\varphi(\cdot,x)](t),$$
for every $(t,x)\in I\times\Omega$ such that $\alpha(x)\in(0,1]$ and the fractional integral of order $\alpha(x)$ exists. By using the associativity of convolution and the semigroup property exhibited by the family of functions $\{g_{\beta}(t):\beta\geq0\}$, it follows that if $\alpha_1,\alpha_2:\Omega\rightarrow[0,1]$, then
\begin{equation}\label{convsemi}J_t^{\alpha_1(x)}J_t^{\alpha_2(x)}\varphi(t,x)=J_t^{\alpha_1(x)+\alpha_2(x)}\varphi(t,x),\end{equation}
for every $(t,x)\in I\times\Omega$ where equation \eqref{convsemi} makes sense (cf. \cite[Theorem 3.15]{CaFe1}).

Now, we introduce the first variable-order fractional derivative considered in this paper. It is important to emphasize that this derivative is introduced solely to define the main variable-order fractional derivative used in our theory.

\begin{definition}\label{RLgeral}
The Riemann-Liouville fractional derivative of order $\alpha(x)$ of a function $\varphi:I\times\Omega\rightarrow\mathbb{R}$ is given by
$$D_t^{\alpha(x)}\varphi(t,x):=\dfrac{d}{dt}\left\{J_t^{1-\alpha(x)}\varphi(t,x)ds\right\},$$
for every $(t,x)\in I\times\Omega$ such that the right side of the above identity exists. More explicitly, we may write
$$D_t^{\alpha(x)}\varphi(t,x)=\left\{\begin{array}{ll}\varphi(t,x),&\textrm{if }\alpha(x)=0,\vspace*{0.2cm}\\\dfrac{d}{dt}\left[\dfrac{1}{\Gamma\big(1-\alpha(x)\big)}\displaystyle\int_0^t(t-s)^{-\alpha(x)}\varphi(s,x)\,ds\right],&\textrm{if }\alpha(x)\in(0,1),\vspace*{0.2cm}\\\varphi_t(t,x),&\textrm{if }\alpha(x)=1.\end{array}\right.$$
\end{definition}

As previously presented in this manuscript, if we consider functions $\alpha:[0,1]\rightarrow[0,1]$, defined as $\alpha(x)=x$, and $\phi:[0,\infty)\times[0,1]\rightarrow\mathbb{R}$, defined as $\phi(t,x)=t^{1/2}\cos(x)$, we can deduce that
$$D_t^{\alpha(x)}\phi(t,x):=\left\{\begin{array}{ll}
t^{1/2}&\textrm{, if }x=0\textrm{ and }t\in[0,\infty),\vspace*{0.3cm}\\
\dfrac{\Gamma(3/2)t^{(1/2)-x}\cos(x)}{\Gamma\big((3/2)-x\big)}&\textrm{, if }x\in(0,1/2]\textrm{ and }t\in[0,\infty),\vspace*{0.3cm}\\
\dfrac{\Gamma(3/2)t^{(1/2)-x}\cos(x)}{\Gamma\big((3/2)-x\big)}&\textrm{, if }x\in(1/2,1]\textrm{ and }t\in(0,\infty),\vspace*{0.3cm}\\
\dfrac{t^{-1/2}\cos(1)}{2}&\textrm{, if }x=1\textrm{ and }t\in(0,\infty).\end{array}\right.$$

\begin{remark} We may verify that the fractional derivative presented in Definition \ref{RLgeral} may have a singularity as $t$ approaches $0^{+}$, similar to the classical Riemann-Liouville fractional derivative, as illustrated in the preceding example.
\end{remark}

We are now ready to introduce our main notion of fractional derivative: the Caputo fractional derivative of variable order.

\begin{definition}
The Caputo fractional derivative of order $\alpha(x)$ of $\varphi:I\times\Omega\rightarrow\mathbb{R}$ is given by
$$cD_t^{\alpha(x)}\varphi(t,x):=D_t^{\alpha(x)}\Big[\varphi(t,x)-\varphi(0,x)\Big]ds,$$
for every $(t,x)\in I\times\Omega$ such that the right side of the above identity exists.
\end{definition}

We conclude this section by presenting some auxiliary results that are recursively used in our manuscript. We begin with two lemmas that aid us in establishing item $(iv)$ of Theorem \ref{espfunc02}, which is a key inequality used in the core result of Section \ref{linearprob}.

\begin{lemma}\label{auxpositive} For $0<\theta<\gamma<1$, we have that
\begin{equation}\label{auxinequality01}\left(\dfrac{1-\theta}{\theta}\right)y^\theta+\left(\dfrac{1-\gamma}{\gamma}\right)y^\gamma-\dfrac{y}{e\log(y)}>0,\end{equation}
for every $y\in\big[e^{1/(1-\theta)},e^{1/(1-\gamma)}\big]$.
\end{lemma}

\begin{proof} 
First, let $x = \log(y)$ and observe that \eqref{auxinequality01} is equivalent to
\begin{equation}\label{auxinequality02}
\left(\dfrac{1-\theta}{\theta}\right)e^{x(\theta-1)+1}x
+\left(\dfrac{1-\gamma}{\gamma}\right)e^{x(\gamma-1)+1}x-1>0,
\end{equation}
for every $x \in \left[{1}/(1-\theta), {1}/(1-\gamma)\right]$. Since for every $x$ in this interval there exists $t \in [0,1]$ such that
$$
x=\left(\dfrac{t}{1-\theta}\right)+\left(\dfrac{1-t}{1-\gamma}\right),
$$
we can rewrite \eqref{auxinequality02} as
\begin{equation}\label{auxinequality03}
\underbrace{\left[\dfrac{(1-\theta)-(\gamma-\theta)t}{\theta(1-\gamma)}\right]e^{-(1-t)(\gamma-\theta)/(1-\gamma)}
+\left[\dfrac{(1-\theta)-(\gamma-\theta)t}{\gamma(1-\theta)}\right]e^{t(\gamma-\theta)/(1-\theta)}-1}_{=:\varphi(t)},
\end{equation}
with $\varphi(t) > 0$ for every $t \in [0,1].$ Therefore, to prove this lemma, we only need to verify that \eqref{auxinequality03} holds for every $t\in[0,1]$. To do so, let us analyze the function $\varphi(t)$. 

First, note that
\begin{multline*}\varphi^\prime(t)=\left[\dfrac{(\gamma-\theta)^2}{\theta\gamma(1-\theta)^2(1-\gamma)^2}\right]\Big[\underbrace{e^{-(1-t)(\gamma-\theta)/(1-\gamma)}\gamma(1-\theta)^2(1-t)}_{=:\psi_1(t)}\\
\underbrace{-e^{t(\gamma-\theta)/(1-\theta)}\theta(1-\gamma)^2t}_{=:\psi_2(t)}\Big],\end{multline*}
for every $t \in [0,1].$ Now, since $\psi_1(t)-\psi_2(t)\geq0$ if and only if 
$$e^{t(\gamma-\theta)^2/[(1-\gamma)(1-\theta)]}\geq\left[\dfrac{e^{(\gamma-\theta)/(1-\gamma)}\theta(1-\gamma)^2}{\gamma(1-\theta)^2}\right]\left(\dfrac{t}{1-t}\right),$$
by analyzing the two functions above for $t\in(0,1)$, we conclude that there exists $r_{\theta\gamma}\in(0,1)$ such that $\varphi^\prime(t)\geq0$ for $t\in(0,r_{\theta\gamma})$ and $\varphi^\prime(t)<0$ for $t\in(r_{\theta\gamma},1)$. This implies that $\varphi(t)$ is increasing on $(0,r_{\theta\gamma})$ and decreasing on $(r_{\theta\gamma},1)$. Therefore, we have
$$\varphi(t)\geq\min\{\varphi(0), \varphi(1)\}>0,$$
for every $t \in [0,1]$. This last inequality, combined with \eqref{auxinequality03}, proves \eqref{auxinequality01}.
\end{proof}

Although the last lemma was very challenging to be proved and required significant effort, it is a crucial tool for proving the next result. 

\begin{lemma}\label{lemmahelp01} Assume that $0<\theta<\gamma<1$ and $\beta\in[\theta,\gamma]$. Then it holds that
$$(1-\beta)y^{\beta}\leq\dfrac{(1-\theta)}{\theta}y^\theta+\dfrac{(1-\gamma)}{\gamma}y^\gamma,$$
for every $y\in(0,\infty)$.
\end{lemma}

\begin{proof} For each fixed $y \in (0, \infty),$ consider the function $f_y:[0,\infty) \rightarrow \mathbb{R}$ defined by
$$f_y(t)=\dfrac{(1-\theta)}{\theta}y^\theta+\dfrac{(1-\gamma)}{\gamma}y^\gamma-(1-t)y^{t}.$$

To present all the details of this proof, let us split it into three cases. \vspace*{0.3cm}

\underline{Case 1:} For $y \in (0, 1)$, observe that $(\log(y)-1)/\log(y)>1$, and $f_y(t)$ is an increasing function in the interval $\big(0,(\log(y)-1)/\log(y)\big)$. Therefore, since $[\theta, \gamma] \subset (0,1)$ and $\beta \in [\theta, \gamma]$, we can infer that
$$f_y(\beta)\geq f_y(\theta)=\dfrac{(1-\theta)^2}{\theta}y^\theta+\dfrac{(1-\gamma)}{\gamma}y^\gamma\geq0,$$
which is equivalent to the desired inequality.\vspace*{0.3cm}

\underline{Case 2:} For $y\in[1,e]$, we note that $f_y(t)$ is an increasing function in $(0,\infty)$. Therefore, considering that $[\theta, \gamma]\subset(0,\infty)$ and $\beta \in [\theta, \gamma]$, we conclude that
$$f_y(\beta)\geq f_y(\theta)=\dfrac{(1-\theta)^2}{\theta}y^\theta+\dfrac{(1-\gamma)}{\gamma}y^\gamma\geq0,$$
which is the desired inequality.\vspace*{0.3cm}

\underline{Case 3:} For $y \in (e, \infty)$, we note that $0 < (\log(y)-1)/\log(y) < 1$. Also observe that function $f_y(t)$ is a decreasing function in $\big(0,(\log(y)-1)/\log(y)\big)$ and an increasing function in $\big((\log(y)-1)/\log(y),\infty\big)$. With this understanding, in the situation where $[\theta, \gamma] \subset \big(0,(\log(y)-1)/\log(y)\big)$, we can deduce that
$$f_y(\beta)\geq f_y(\gamma)=\dfrac{(1-\theta)}{\theta}y^\theta+\dfrac{(1-\gamma)^2}{\gamma}y^\gamma\geq0,$$
yielding the desired inequality. 

In the case where $[\theta, \gamma] \subset \big((\log(y)-1)/\log(y),\infty\big)$, we can deduce that
$$f_y(\beta)\geq f_y(\theta)=\dfrac{(1-\theta)^2}{\theta}y^\theta+\dfrac{(1-\gamma)}{\gamma}y^\gamma\geq0,$$
again providing the desired inequality. Finally, if $(\log(y)-1)/\log(y) \in [\theta, \gamma]$, we can deduce that
$$f_y(\beta)\geq f_y\left(\dfrac{\log(y)-1}{\log(y)}\right)=\underbrace{\dfrac{(1-\theta)}{\theta}y^\theta+\dfrac{(1-\gamma)}{\gamma}y^\gamma-\dfrac{y}{e\log(y)}}_{=:g(y)},$$
and $y\in [e^{1/(1-\theta)},e^{1/(1-\gamma)}]$. Since Lemma \ref{auxpositive} ensures that $g(y)>0$ in the interval $[e^{1/(1-\theta)},e^{1/(1-\gamma)}]$, we have once again established the desired inequality.
\end{proof}

We can now use Lemma \ref{lemmahelp01} to derive our first main result, which provides essential inequalities necessary for proving the first theorem of existence and uniqueness of solution (cf. Theorem \ref{exisuni}) we present in this work.

\begin{theorem}\label{espfunc02} Let $\alpha_0\in(0,1)$ and assume that $\alpha\in L^\infty(\Omega)$ is such that $\alpha(x)\in[\alpha_0,1]$ for almost every $x\in\Omega$. The following inequalities hold:\vspace*{0.3cm}
\begin{itemize}
\item[(i)] Let $\eta=1.4616\ldots$ be the minimum of Euler's Gamma function. Then  
$$\Gamma(\eta)<\Gamma(\|\alpha\|_{L^\infty(\Omega)})\leq\Gamma(\alpha(x))\leq\Gamma(\alpha_0),$$
for almost every $x\in\Omega.$\vspace*{0.3cm}
\item[(ii)] If $0<s\leq1$ we have that
$${s^{\|\alpha\|_{L^\infty(\Omega)}-1}}\leq{s^{\alpha(x)-1}}\leq{s^{\alpha_0-1}},$$
for almost every $x\in\Omega$.\vspace*{0.3cm}
\item[(iii)] If $s\geq1$ it holds that
$${s^{\alpha_0-1}}\leq{s^{\alpha(x)-1}}\leq{s^{\|\alpha\|_{L^\infty(\Omega)}-1}},$$
for almost every $x\in\Omega$.\vspace*{0.3cm}
\item[(iv)] If $s_1,s_2>0$ we have
$$\left|{s_2^{\alpha(x)-1}-s_1^{\alpha(x)-1}}\right|\leq\dfrac{\left|s_2^{\alpha_0-1}-s_1^{\alpha_0-1}\right|}{\alpha_0}+\dfrac{\left|s_2^{\|\alpha\|_{L^\infty(\Omega)}-1}-s_1^{\|\alpha\|_{L^\infty(\Omega)}-1}\right|}{\|\alpha\|_{L^\infty(\Omega)}},$$
for almost every $x\in\Omega$.
\end{itemize}
\end{theorem}

\begin{proof} We establish items $(i)$, $(ii)$, and $(iii)$ through the following observations: $\alpha_0 \leq \alpha(x) \leq \|\alpha\|_{L^\infty(\Omega)}$ for almost every $x\in\Omega$, and the Gamma function is decreasing in the interval $(0, \eta)$; for details, see \cite{CoDe1}.

To prove item $(iv)$, a more detailed explanation is necessary. Assuming $0 < \theta < \beta < \gamma < 1$, consider the function $g:(0,\infty)\rightarrow\mathbb{R}$ defined by
$$g(t)=1-\dfrac{(\theta-1)t^{\theta-\beta}}{(\beta-1)\theta}-\dfrac{(\gamma-1)t^{\gamma-\beta}}{(\beta-1)\gamma}.$$
Observe that $g(t)$ is an increasing function in the interval $(0,{t_*})$ and a decreasing function in the interval $({t_*},\infty)$, where
$${t_*}=\left[\dfrac{(1-\theta)(\beta-\theta)\gamma}{(1-\gamma)(\gamma-\beta)\theta}\right]^{1/(\gamma-\theta)}.$$
But then, for every $t \in (0, \infty)\setminus\{t_*\},$ it holds that $g(t) < g(t_*)$. However, Lemma \ref{lemmahelp01} ensures that
$$(1-\beta){t_*}^{\beta}\leq\dfrac{(1-\theta)}{\theta}{t_*}^\theta+\dfrac{(1-\gamma)}{\gamma}{t_*}^\gamma,$$
which is equivalent to $g({t_*})\leq0.$ But this ensures that $g(t) \leq 0$, for every $t \in (0, \infty)$. Hence, if we consider $f:(0,\infty)\rightarrow\mathbb{R}$ given by
$$f(t)=t^{\beta-1}-\dfrac{t^{\theta-1}}{\theta}-\dfrac{t^{\gamma-1}}{\gamma},$$
since $f^\prime(t)=(\beta-1)t^{\beta-2}g(t),$ we may conclude that $f(t)$ is non decreasing in $(0,\infty)$. Consequently, if we consider, without loss of generality, that $s_1<s_2$, then $f(s_1)\leq f(s_2)$, and therefore
\begin{multline}\label{finaleq}\left|{s_1^{\beta-1}-s_2^{\beta-1}}\right|={s_1^{\beta-1}-s_2^{\beta-1}}=f(s_1)-f(s_2)+\dfrac{s_1^{\theta-1}-s_2^{\theta-1}}{\theta}+\dfrac{s_1^{\gamma-1}-s_2^{\gamma-1}}{\gamma}\\
\leq\dfrac{s_1^{\theta-1}-s_2^{\theta-1}}{\theta}+\dfrac{s_1^{\gamma-1}-s_2^{\gamma-1}}{\gamma}=\left|\dfrac{s_1^{\theta-1}-s_2^{\theta-1}}{\theta}\right|
+\left|\dfrac{s_1^{\gamma-1}-s_2^{\gamma-1}}{\gamma}\right|.\end{multline}
With a simple limit argument, we can prove that \eqref{finaleq} holds even when $0 < \theta \leq \beta \leq \gamma \leq 1$. Thus, by taking $\theta=\alpha_0$, $\beta=\alpha(x)$, and $\gamma=\|\alpha\|_{L^\infty(\Omega)}$, we complete the proof of item $(iv)$.
\end{proof}

\begin{remark} If we assume that $\|\alpha\|_\infty<1$, a simpler proof than the one presented above allows us to establish that for $s_1, s_2 > 0$,
$$\left|{s_2^{\alpha(x)-1}-s_1^{\alpha(x)-1}}\right|\leq\dfrac{\left|s_2^{\alpha_0-1}-s_1^{\alpha_0-1}\right|}{1-\alpha_0}
+\dfrac{\left|s_2^{\|\alpha\|_{L^\infty(\Omega)}-1}-s_1^{\|\alpha\|_{L^\infty(\Omega)}-1}\right|}{1-\|\alpha\|_{L^\infty(\Omega)}},$$
for almost every $x\in\Omega$. However, since this result neither strengthens our main theorems nor allows us to consider the case $\|\alpha\|_{L^\infty(\Omega)}=1$, we choose to mention it here without providing a detailed proof.

\end{remark}

Another important result we present, concerns the continuity of the Riemann-Liouville fractional integral of order $\alpha(x)$ in $L^p$ spaces, as described below.
\begin{corollary}\label{contvar} If $\alpha\in L^\infty(\Omega)$ and there exists $\alpha_0\in(0,1)$ such that $\alpha(x)\in[\alpha_0,1]$ for almost every $x\in\Omega$, and $\varphi\in L^r(0,T;L^p(\Omega))$, for $1\leq r,p\leq\infty$, then $J_t^{\alpha(x)}\varphi(t,x)$ belongs to $L^r(0,T;L^p(\Omega))$ and
\begin{multline*}\left(\int_0^T\left\|\dfrac{1}{\Gamma(\alpha(\cdot))}\int_{0}^t{(t-s)^{\alpha(\cdot)-1}\varphi(s,\cdot)}\,ds\right\|^r_{L^p(\Omega)}\,dt\right)^{1/r}
\\\leq\dfrac{1}{\Gamma(\|\alpha\|_{L^\infty(\Omega)})}\left[\dfrac{T^{\alpha_0}}{\alpha_0}+\dfrac{T^{\|\alpha\|_{L^\infty(\Omega)}}}{\|\alpha\|_{L^\infty(\Omega)}}\right]\|\varphi\|_{L^r(0,T;L^p(\Omega))}.\end{multline*}

\end{corollary}

\begin{proof} Given that $\alpha \in L^\infty(\Omega)$, $t^{\alpha(x)-1}$ is a Lebesgue measurable function for every $(t, x) \in [0, T] \times \Omega$. Therefore, since $\phi(t)$ is a Bochner-Lebesgue measurable function, $(t-s)^{\alpha(\cdot)-1} \varphi(s, \cdot)$ is a Bochner-Lebesgue measurable function for almost every $t \in [0, T]$ (cf. \cite[Chapter XIV - Lemma 1.1]{Mi1}).

Now, due to items $(i)$, $(ii)$ and $(iii)$ of Theorem \ref{espfunc02}, observe that
\begin{multline*}
\left|\dfrac{1}{\Gamma(\alpha(x))}\int_{0}^t{(t-s)^{\alpha(x)-1}\varphi(s,x)}\,ds\right|\\
\leq\dfrac{1}{\Gamma(\|\alpha\|_{L^\infty(\Omega)})}\int_{0}^t{\Big[(t-s)^{\alpha_0-1}+(t-s)^{\|\alpha\|_{L^\infty(\Omega)}-1}\Big]|\varphi(s,x)|}\,ds,
\end{multline*}
for almost every $(t, x) \in [0, T] \times \Omega$, which implies, thanks to Minkowski's integral inequality (cf. \cite[Theorem 202]{HaLiPo1}), that
\begin{multline*}\|J_t^{\alpha(\cdot)}\varphi(t,\cdot)\|_{L^p(\Omega)}
\\\leq\dfrac{1}{\Gamma(\|\alpha\|_{L^\infty(\Omega)})}\Big[\Gamma(\alpha_0)J_t^{\alpha_0}\|\varphi(t)\|_{L^p(\Omega)}
+\Gamma(\|\alpha\|_{L^\infty(\Omega)})J_t^{\|\alpha\|_{L^\infty(\Omega)}}\|\varphi(t)\|_{L^p(\Omega)}\Big],\end{multline*}
for almost every $t \in [0, T]$. Hence, for $1\leq r<\infty$ we have that
\begin{multline*}\left(\int_0^T\left\|\dfrac{1}{\Gamma(\alpha(\cdot))}\int_{0}^t{(t-s)^{\alpha(\cdot)-1}\varphi(s,\cdot)}\,ds\right\|^r_{L^p(\Omega)}\,dt\right)^{1/r}
\\\leq\dfrac{1}{\Gamma(\|\alpha\|_{L^\infty(\Omega)})}\Big(\int_0^T\Big[\Gamma(\alpha_0)J_t^{\alpha_0}\|\varphi(t)\|_{L^p(\Omega)}\\
+\Gamma(\|\alpha\|_{L^\infty(\Omega)})J_t^{\|\alpha\|_{L^\infty(\Omega)}}\|\varphi(t)\|_{L^p(\Omega)}\Big]^r\,dt\Big)^{1/r}.\end{multline*}

Finally, Minkowski's classical inequality and the boundedness of Riemann-Liouville fractional integral (see \cite[Theorem 3.1]{CaFe1} for details) ensure that
\begin{multline*}\left(\int_0^T\left\|\dfrac{1}{\Gamma(\alpha(\cdot))}\int_{0}^t{(t-s)^{\alpha(\cdot)-1}\varphi(s,\cdot)}\,ds\right\|^r_{L^p(\Omega)}\,dt\right)^{1/r}
\\\leq\dfrac{1}{\Gamma(\|\alpha\|_{L^\infty(\Omega)})}\left[\dfrac{T^{\alpha_0}}{\alpha_0}+\dfrac{T^{\|\alpha\|_{L^\infty(\Omega)}}}{\|\alpha\|_{L^\infty(\Omega)}}\right]\|\varphi\|_{L^r(0,T;L^p(\Omega))}.\end{multline*}

The case $r = \infty$ is omitted here, as it can be verified straightforwardly.
\end{proof}

\begin{remark} Some points that we find important to emphasize:
\begin{itemize}
\item[(i)] Corollary \ref{contvar} ensures that the Riemann-Liouville fractional integral of order $\alpha(x)$, when $\alpha \in L^\infty(\Omega)$ and is bounded below by $\alpha_0\in(0,1)$, is a bounded linear operator from $L^r(0, T; L^p(\Omega))$ into $L^r(0, T; L^p(\Omega))$ for every $1 \leq r,p \leq \infty$. Therefore, we are improving Theorem 3.1 from \cite{CaFe1} for the case when $X=L^p(\Omega)$.\vspace{0.2cm}
\item[(ii)] If $\alpha \in L^\infty(\Omega)$ and is bounded below by $\alpha_0 \in (0,1)$, and if $T > 0$ is such that $\varphi \in L^1(0, T; L^p(\Omega))$ and $J_t^{1 - \alpha(x)} \varphi(t, x)$ belongs to the Bochner-Sobolev space $W^{1,1}(0, T; L^p(\Omega))$, then the Riemann-Liouville fractional derivative of order $\alpha(x)$ of $\varphi(t, x)$ exists for almost every $(t, x) \in [0, T] \times \Omega$. Furthermore, if $\varphi \in C([0, T]; L^p(\Omega))$, then the Caputo fractional derivative of order $\alpha(x)$ of $\varphi(t, x)$ also exists for almost every $(t, x) \in [0, T] \times \Omega$. \vspace{0.2cm}
\item[(iii)] When $\varphi \in L^1(0, \infty; L^p(\Omega))$, $\alpha \in L^\infty(\Omega)$ is bounded below by $\alpha_0 \in (0,1)$, then $J_t^{\alpha(x)} \varphi(t, x)$ does not generally belong to $L^1(0, \infty; L^p(\Omega))$. For example, consider $\alpha:[0,1]\rightarrow\mathbb{R}$ given by $\alpha(x)=1/2$, and $\varphi:[0,\infty)\times[0,1]\rightarrow\mathbb{R}$ given by 
    $$\varphi(t,x)=\left\{\begin{array}{ll}1,&\textrm{ if }0\leq t\leq1,\vspace*{0.2cm}\\
    0,&\textrm{ if }t>1,\end{array}\right.$$
    Then, for $1 \leq p \leq \infty$, $\varphi \in L^1(0, \infty; L^p(0,1))$. However, since
    $$J_t^{\alpha(x)}\varphi(t,x)=\left\{\begin{array}{ll}\dfrac{t^{1/2}}{\Gamma(3/2)},&\textrm{ if }0\leq t\leq1,\vspace*{0.2cm}\\
    \dfrac{t^{1/2}-(t-1)^{1/2}}{\Gamma(3/2)},&\textrm{ if }t>1,\end{array}\right.$$
    and
    $$\big\|J_t^{\alpha(x)}\varphi(t,x)\big\|_{L^1(0, \infty; L^p(0,1))}=\lim_{M\rightarrow\infty}\left[\dfrac{M^{3/2}-(M-1)^{3/2}}{\Gamma(5/2)}\right]=\infty,$$
    we conclude that $J_t^{\alpha(x)}\varphi(t,x)$ does not belongs to $\varphi\in L^1(0, \infty; L^p(0,1))$.
\end{itemize}
\end{remark}

Lastly, we present two important results concerning the Riemann-Liouville fractional integral and derivative of order $\alpha(x)$. These properties are essential for the analysis of solutions discussed in Section \ref{linearprob}.
\begin{proposition}\label{auxlemmaprev} Let $\alpha_0\in(0,1)$ and assume that $\alpha\in L^\infty(\Omega)$ is such that $\alpha(x)\in[\alpha_0,1]$ for almost every $x\in\Omega$. For $T>0$, $1\leq p\leq\infty$ and $\varphi\in C([0,T];L^p(\Omega))$, it holds that $J^{\alpha(x)}_t\varphi(t,x)$ belongs to $C([0,T];L^p(\Omega))$ and that $J^{\alpha(x)}_t\varphi(t,x)\big|_{t=0}=0$ for almost every $x\in\Omega$.
\end{proposition}

\begin{proof} Since $\alpha(x)\in(\alpha_0,1]$, if we consider $t\in(0,T]$ and use Hölder's inequality along with items $(i)-(iii)$ of Theorem \ref{espfunc02}, we obtain
\begin{multline*}
\left\|\dfrac{1}{\Gamma(\alpha(\cdot))}\int_{0}^t{(t-s)^{\alpha(\cdot)-1}\varphi(s,\cdot)}\,ds\right\|_{L^p(\Omega)}\leq
\int_{0}^t{\left\|\dfrac{(t-s)^{\alpha(\cdot)-1}}{\Gamma(\alpha(\cdot))}\varphi(s,\cdot)\right\|_{L^p(\Omega)}}\,ds\\
\leq \|\varphi\|_{C([0,T];L^p(\Omega))}\int_{0}^t{\left\|\dfrac{(t-s)^{\alpha(\cdot)-1}}{\Gamma(\alpha(\cdot))}\right\|_{L^\infty(\Omega)}}\,ds\\
\leq \|\varphi\|_{C([0,T];L^p(\Omega))}\left[\dfrac{t^{\|\alpha\|_\infty}}{\|\alpha\|_\infty}+\dfrac{t^{\alpha_0}}{\alpha_0}\right]\dfrac{1}{\Gamma(\|\alpha\|_\infty)}.
\end{multline*}
Thus, the expression is well-defined for every $t\in[0,T]$ and is continuous at $t=0$.

Let us now show that it is continuous in $(0,T]$. Fix $t_0\in(0,T]$. If $t \in (t_0, T)$, we can deduce
\begin{multline*}
\left\|\dfrac{1}{\Gamma(\alpha(\cdot))}\int_0^t(t-s)^{\alpha(\cdot)-1}\varphi(s,\cdot)\,ds-\dfrac{1}{\Gamma(\alpha(\cdot))}\int_0^{t_0}(t_0-s)^{\alpha(\cdot)-1}\varphi(s,\cdot)\,ds\right\|_{L^p(\Omega)}\\
\leq\int_{t_0}^t\left\|\dfrac{(t-s)^{\alpha(\cdot)-1}}{\Gamma(\alpha(\cdot))}\varphi(s,\cdot)\right\|_{L^p(\Omega)}\,ds\\
+\int_0^{t_0}\left\|\left[\dfrac{(t_0-s)^{\alpha(\cdot)-1}-(t-s)^{\alpha(\cdot)-1}}{\Gamma(\alpha(\cdot))}\right]\varphi(s,\cdot)\right\|_{L^p(\Omega)}\,ds.
\end{multline*}
Therefore, by applying Hölder's inequality and performing a change of variables, we obtain
\begin{multline*}
\left\|\dfrac{1}{\Gamma(\alpha(\cdot))}\int_0^t(t-s)^{\alpha(\cdot)-1}\varphi(s,\cdot)\,ds-\dfrac{1}{\Gamma(\alpha(\cdot))}\int_0^{t_0}(t_0-s)^{\alpha(\cdot)-1}\varphi(s,\cdot)\,ds\right\|_{L^p(\Omega)}\\
\leq\|\varphi\|_{C([0,T];L^p(\Omega))}\left[\int_{0}^{t-t_0}\left\|\dfrac{w^{\alpha(\cdot)-1}}{\Gamma(\alpha(\cdot))}\right\|_{L^\infty(\Omega)}\,dw
\right.\\\left.+\int_0^{t_0}\left\|\left[\dfrac{w^{\alpha(\cdot)-1}-(t-t_0+w)^{\alpha(\cdot)-1}}{\Gamma(\alpha(\cdot))}\right]\right\|_{L^\infty(\Omega)}\,dw\right].
\end{multline*}

Hence, items $(i)-(iv)$ of Theorem \ref{espfunc02} ensure
\begin{multline*}
\left\|\dfrac{1}{\Gamma(\alpha(\cdot))}\int_0^t(t-s)^{\alpha(\cdot)-1}\varphi(s,\cdot)\,ds-\dfrac{1}{\Gamma(\alpha(\cdot))}\int_0^{t_0}(t_0-s)^{\alpha(\cdot)-1}\varphi(s,\cdot)\,ds\right\|_{L^p(\Omega)}\\
\leq\dfrac{\|\varphi\|_{C([0,T];L^p(\Omega))}}{\Gamma(\|\alpha\|_\infty)}\left[\int_{0}^{t-t_0}\big(w^{\alpha_0-1}+w^{\|\alpha\|_\infty-1}\big)\,dw
\right.\\\left.+\int_0^{t_0}\dfrac{\big(w^{\alpha_0-1}-(t-t_0+w)^{\alpha_0-1}\big)}{\alpha_0}+\dfrac{\big(w^{\|\alpha\|_\infty-1}-(t-t_0+w)^{\|\alpha\|_\infty-1}\big)}{\|\alpha\|_\infty}\,dw\right].
\end{multline*}
Finally we achieve the inequality
\begin{multline*}
\left\|\dfrac{1}{\Gamma(\alpha(\cdot))}\int_0^t(t-s)^{\alpha(\cdot)-1}\varphi(s,\cdot)\,ds-\dfrac{1}{\Gamma(\alpha(\cdot))}\int_0^{t_0}(t_0-s)^{\alpha(\cdot)-1}\varphi(s,\cdot)\,ds\right\|_{L^p(\Omega)}\\
\leq\dfrac{\|\varphi\|_{C([0,T];L^p(\Omega))}}{\Gamma(\|\alpha\|_\infty)}\left[\dfrac{(1+\|\alpha\|_\infty)(t-t_0)^{\|\alpha\|_\infty}+{t_0}^{\|\alpha\|_\infty}-t^{\|\alpha\|_\infty}}{\|\alpha\|^2_\infty}\right.\\
\left.+\dfrac{(1+\alpha_0)(t-t_0)^{\alpha_0}+{t_0}^{\alpha_0}-{t}^{\alpha_0}}{\alpha_0^2}\right].
\end{multline*}
It is not difficult to observe that for $t \in (0,t_0)$, we can obtain a similar estimate. Therefore, we conclude the continuity of $J^{\alpha(\cdot)}_t\varphi(t,\cdot)$ with respect to $t$, as desired.

\end{proof}

\begin{remark} Note that the constraint $\alpha(x)\geq\alpha_0>0$ is necessary for the proof of Proposition \ref{auxlemmaprev} to be valid, since otherwise, it would not be possible to apply Theorem \ref{espfunc02}, which would compromise the fact that
{$$\displaystyle \lim_{t \rightarrow0}\int_0^t \left\| \dfrac{(t - s)^{\alpha(x) - 1}}{\Gamma(\alpha(x))} \right\|_{L^\infty(\Omega)}\,ds=0.$$}
Consider, for example, $\Omega=(0,1]\subset\mathbb{R}$, $I=[0,e^{\psi(1)}]\subset\mathbb{R}$, and $\alpha(x)=x$. Here, the function $\psi(x)$ represents the Digamma function (for details on the Digamma function we refer to \cite{AbSt1}). Let $t\in(0,e^{\psi(1)}]$, and observe that
\begin{equation}\label{helpnew01}\int_0^t \left\| \dfrac{(t - s)^{\alpha(x) - 1}}{\Gamma(\alpha(x))} \right\|_{L^{\infty}(\Omega)}ds = \int_0^t \left\| \dfrac{s^{x - 1}}{\Gamma(x)} \right\|_{L^{\infty}(\Omega)}ds.\end{equation}
Also note that for each $s\in(0,t)\subset(0,e^{\psi(1)})$, we have
$$\dfrac{d}{dx}\left(\dfrac{s^{x - 1}}{\Gamma(x)}\right)=\left[\dfrac{s^{x-1}}{\Gamma(x)}\right]\left[{\log(s)-\psi(x)}\right].$$
Recalling that the Digamma function is increasing in $(0,\infty)$ and satisfies
$$\lim_{y\rightarrow0^+}\psi(y)=-\infty\quad\textrm{ and }\quad\lim_{y\rightarrow\infty}\psi(y)=\infty,$$
we deduce the existence of a unique value $x_s\in (0,1)$ such that $\log s=\psi(x_s)$. Hence, we conclude that $(0,1)\ni x\mapsto s^{x - 1}/\Gamma(x)$ is increasing on $(0,x_s)$ and decreasing on $(x_s,1)$, what allows us to compute that
$$\left\| \dfrac{s^{x - 1}}{\Gamma(x)} \right\|_{L^{\infty}(\Omega)}=\dfrac{s^{x_s - 1}}{\Gamma(x_s)}=\dfrac{s^{\psi^{-1}(\log(s))-1}}{\Gamma(\psi^{-1}(\log(s)))}.$$

This, together with \eqref{helpnew01}, leads to
$$\displaystyle \int_0^t \left\| \dfrac{(t - s)^{x - 1}}{\Gamma(x)} \right\|_{L^\infty(\Omega)}\,ds = \displaystyle \int_0^t \dfrac{s^{\psi^{-1}(\ln(s))-1}}{\Gamma(\psi^{-1}(\ln(s)))}\,ds=\int_0^{\psi^{-1}(\ln(t))}\dfrac{e^{w\psi(w)}\psi^\prime(w)}{\Gamma(w)}\,dw.
$$
Now observe that:
\begin{itemize}
\item[(i)] since $t\in(0,e^{\psi(1)})$, we have that $\psi^{-1}(\ln(t))\in(0,1)$;\vspace*{0.2cm}
\item[(ii)] for $w\in(0,1)$ we have the lower bound $w\psi(w)\geq -e^{-1}-1$ (see \cite[Lemma 1]{GuQi1} for details); \vspace*{0.2cm}
\item[(iii)] for $w\in(0,1)$ we have the lower bound $\psi^\prime(w)\geq 1/(2w^2)$ (see \cite[Lemma 1]{GuQi1} for details); \vspace*{0.2cm}
\item[(iv)] for $w\in(0,1)$ we have the upper bound $\Gamma(w+1)\leq (w^2+2)/(w+2)$ (see \cite{AlKw1} for details).\vspace*{0.2cm}
\end{itemize}

Items $(i)-(iv)$ are enough for us to conclude that
$$\displaystyle \int_0^t \left\| \dfrac{(t - s)^{x - 1}}{\Gamma(x)} \right\|_{L^\infty(\Omega)}\,ds \geq \dfrac{e^{(-e^{-1}-1)}}{2}\int_0^{\psi^{-1}(\ln(t))}\dfrac{w+2}{w^3+2w}\,dw=\infty,
$$
for every $t\in(0,e^{\psi(1)})$.
\end{remark}

\begin{proposition}\label{help01}  Let $\alpha_0\in(0,1)$ and assume that $\alpha\in L^\infty(\Omega)$ is such that $\alpha(x)\in[\alpha_0,1]$ for almost every $x\in\Omega$. For $T>0$, $1\leq p\leq\infty$ and $\varphi(t,x)\in C([0,T];L^p(\Omega))$, we have:\vspace*{0.2cm}
    \begin{itemize}
    \item [(i)]  $cD_t^{{\alpha(x)}}\left[J_t^{{\alpha(x)}}\varphi(t,x)\right]=\varphi(t,x),$ for every $t\in [0,T]$ and almost every $x\in\Omega$.\vspace*{0.2cm}
    \item[(ii)] If $\varphi\in C^1([0,T];L^p(\Omega))$, then $cD_t^{\alpha(x)}\varphi(t,x)$ belongs to $C([0,T];L^p(\Omega))$ and satisfies %
        $$cD_t^{\alpha(x)}\varphi(t,x)=J_t^{1-\alpha(x)}\varphi^\prime(t,x),$$
        for every $t\in [0,T]$ and almost every $x\in\Omega$.\vspace*{0.2cm}
    \item [(iii)] If $J_t^{1-\alpha(x)}[\varphi(t,x)-\varphi(0,x)]$ belongs to $C^1([0,T];L^p(\Omega))$, then we may conclude that $cD_t^{\alpha(x)}\varphi(t,x)$ belongs to $C([0,T];L^p(\Omega))$. Moreover, it holds that
    $$J_t^{{\alpha(x)}}\Big[cD_t^{{\alpha(x)}}\varphi(t)\Big]=\varphi(t,x)-\varphi(0,x),$$
    for every $t\in [0,T]$ and almost every $x\in\Omega$.\vspace*{0.2cm}
    \end{itemize}
  \end{proposition}
  \begin{proof} $(i)$ As a consequence of Proposition \ref{auxlemmaprev}, we know that $J_t^{\alpha(x)}\varphi(t,x)$ exists for every $t\in [0,T]$ and almost every $x\in\Omega$, belongs to $C([0,T];L^p(\Omega))$ and satisfies $J_s^{{\alpha(x)}}\varphi(s,x)\big|_{s=0}=0$ for almost every $x\in\Omega$. Therefore, also exists $J_t^{1-\alpha(x)}\Big[J_t^{\alpha(x)}\varphi(t,x)\Big]$ and belongs to $C([0,T];L^p(\Omega))$. Then, from equation \eqref{convsemi}, it follows that
\begin{multline*}
cD_t^{{\alpha(x)}}\left[J_t^{{\alpha(x)}}\varphi(t,x)\right]=D_t^{{\alpha(x)}}\Big[J_t^{{\alpha(x)}}\varphi(t,x)-J_s^{{\alpha(x)}}\varphi(s,x)\big|_{s=0}\Big]
  \\=\dfrac{d}{dt}\Big\{J_t^{1-\alpha(x)}\left[J_t^{{\alpha(x)}}\varphi(t,x)\right]\Big\}=\varphi(t,x),
\end{multline*}
for every $t\in [0,T]$ and almost every $x\in\Omega$ as desired.\vspace*{0.2cm}

  $(ii)$ First observe that
  \begin{multline*}
  J_t^{{1-\alpha(x)}}\big[\varphi(t,x)-\varphi(0,x)\big]=\dfrac{1}{\Gamma(1-\alpha(x))}\int_0^t(t-s)^{-\alpha(x)}\big[\varphi(s,x)-\varphi(0,x)\big]ds\\
  =-\dfrac{1}{\Gamma(2-\alpha(x))}\int_0^t\dfrac{d}{ds}(t-s)^{1-\alpha(x)}\big[\varphi(s,x)-\varphi(0,x)\big]ds\\
  =\dfrac{1}{\Gamma(2-\alpha(x))}\int_0^t(t-s)^{1-\alpha(x)}\varphi^\prime(s,x)ds=J_t^{{2-\alpha(x)}}\varphi^\prime(t,x),
  \end{multline*}
  for every $t\in [0,T]$ and almost every $x\in\Omega$.

  Therefore, by using equation \eqref{convsemi}, we can deduce that
  $$cD_t^\alpha\varphi(t,x)=J_t^{{1-\alpha(x)}}\varphi^\prime(t,x)$$
  for every $t\in [0,T]$ and almost every $x\in\Omega$, what together with Proposition \ref{auxlemmaprev} establishes the desired conclusions.\vspace*{0.2cm}

  $(iii)$ By applying Proposition \ref{auxlemmaprev} along with items $(i)$ and $(ii)$, we can deduce the identity:
  \begin{multline*}
  J_t^{{\alpha(x)}}\left[cD_t^{{\alpha(x)}}\varphi(t,x)\right]
  =J_t^{{\alpha(x)}}\left\{\dfrac{d}{dt}\left[J_t^{{1-\alpha(x)}}\big(\varphi(t,x)-\varphi(0,x)\big)\right]\right\}\\
  =D_t^{1-\alpha(x)}\left[J_t^{{1-\alpha(x)}}\big(\varphi(t,x)-\varphi(0,x)\big)\right]=\varphi(t,x)-\varphi(0,x),
  \end{multline*}
  for every $t\in [0,T]$ and almost every $x\in\Omega$.
  \end{proof}

%%%%%%%%%%%%%%%%%%%%%%%%%%%%%%%%%%%%%%%%%%%%%%%%%%%%%%%%%%%%%%%%%

\section{Well-Posedness of the Linear Problem}
\label{linearprob}

Consider $1\leq p\leq\infty$, $\Omega$ an open subset of $\mathbb{R}^n$ and $A:L^p(\Omega)\rightarrow L^p(\Omega)$ a bounded linear operator. Given $\alpha\in L^\infty(\Omega)$ satisfying $\alpha(x)\in[\alpha_0,1]$, for some $\alpha_0\in(0,1)$, and $u_0\in L^p(\Omega)$, in this section we investigate the Cauchy problem for the space-dependent fractional evolution equation of order $\alpha(x)$ given by
\begin{subequations}\label{eq1011}
  \begin{align}\label{eq101}&cD_t^{\alpha(x)} u(t,x)=Au(t,x),&& \forall t>0\textrm{ and a.e. }x\in\Omega,\\
\label{eq1101}&u(0,x)=u_0(x),&& \textrm{ a.e. }x\in\Omega.
 \end{align}
\end{subequations}
Recall that the symbol $cD_t^{\alpha(x)}$ stands for the Caputo fractional derivative of order $\alpha(x)$ with respect to the variable $t$, as presented in Section \ref{fracintro}.

While this problem may appear straightforward at first, we show that it presents significant challenges and demands a deep understanding before moving on to more complex cases.

\begin{definition}\label{integralequ1} We say that $\varphi\in C(I;L^p(\Omega))$ is solution of \eqref{eq1011} in $I\times\Omega$ if it satisfies:
\begin{itemize}
\item[(i)] $cD_t^{\alpha(x)} \varphi(t,x)=A\varphi(t,x)$ for every $t\in I$ and almost every $x\in\Omega.$\vspace*{0.2cm}
\item[(ii)] $\varphi(0,x)=u_0(x)$ for almost every $x\in\Omega.$
\end{itemize}
\end{definition}

Now we present a general version of the classical theorem that relates the solutions of integral and differential equations.
\begin{proposition}\label{integralequ}
Let $T>0$ and $\varphi\in C([0,T];L^p(\Omega))$. Then $\varphi(t,x)$ is a solution of \eqref{eq1011} in $[0,T]\times\Omega$ if, and only if, $\varphi(t,x)$ satisfies the integral equation
\begin{equation} \label{eqint}
\varphi(t,x)=u_0(x)+\dfrac{1}{\Gamma(\alpha(x))}\int_{0}^t{(t-s)^{\alpha(x)-1}A\varphi(s,x)}\,ds,
\end{equation}
for every $t \in [0,T]$ and almost every $x \in \Omega$.
\end{proposition}
\begin{proof}
Assuming that $\varphi(t,x)$ is a solution of equation \eqref{eq1011}, we can conclude that $A\varphi(t,x)$ belongs to the space $C([0,T];L^p(\Omega))$, since $A\in\mathcal{L}(L^p(\Omega))$. Using this fact and applying the operator $J_t^{\alpha(x)}$ to both sides of equation \eqref{eq101}, along with Propositions \ref{auxlemmaprev} and \ref{help01}, and equation \eqref{eq1101}, we deduce that
\begin{equation*}
\varphi(t,x)=u_0(x)+\dfrac{1}{\Gamma(\alpha(x))}\int_{0}^t{(t-s)^{\alpha(x)-1}A\varphi(s,x)},ds, 
\end{equation*}
forall $t\in[0,T]$ and a.e. $x\in\Omega$.

Conversely, if $\varphi(t,x)$ satisfies \eqref{eqint}, since $A\in\mathcal{L}(L^p(\Omega))$ and therefore $A\varphi(t,x)$ belongs to $C([0,T];L^p(\Omega))$, we can apply $cD_t^\alpha(x)$ on both sides of \eqref{eqint} and use item $(i)$ of Proposition \ref{help01} to obtain \eqref{eq101}.

Finally, Proposition \ref{auxlemmaprev} is enough for us to verify
$$\lim_{t\rightarrow0^+}\|\varphi(t,\cdot)-u_0(\cdot)\|_{L^p(\Omega)}=\lim_{t\rightarrow0^+}\|J_t^{\alpha(\cdot)}\varphi(t,\cdot)\|_{L^p(\Omega)}=0,$$
i.e., that \eqref{eq1101} holds.

Therefore $\varphi(t,x)$ is a solution of the Cauchy problem \eqref{eq1011}.
\end{proof}

Based on the notions and results established thus far in this paper, we are now able to present one of the main results of this section. We should not underestimate the complexity that fractional derivatives and the dependence of the derivative on the spatial variable introduce into the proof of this result. This is the reason we provide a detailed analysis to ensure its understanding.
\begin{theorem}\label{exisuni} The abstract Cauchy problem \eqref{eq1011} has a unique solution in $[0,\infty)\times\Omega$.
\end{theorem}
\begin{proof}
Define the value
$$\tau = \min \left\{ \left( \dfrac{\alpha_0\Gamma(\| \alpha \|_{L^{\infty}(\Omega)})}{2 \| A \|_{\mathcal{L}(L^p(\Omega))}} \right)^{\frac{1}{\alpha_0}} , \dfrac{1}{2} \right\},$$
and observe that we may consider $T: C(\left[ 0, \tau \right]; L^p(\Omega)) \longrightarrow C(\left[ 0, \tau \right]; L^p(\Omega))$ given by:
\begin{equation}\label{mainform01}T\varphi(t,x) := u_0(x) + \displaystyle \int_0 ^{t} \dfrac{(t - s)^{\alpha(x) - 1}}{\Gamma(\alpha(x))}A\varphi(s,x)ds.\end{equation}
It worths to emphasize that $T$ is well-defined thanks to Proposition \ref{auxlemmaprev}. Therefore, for $\varphi_1,\varphi_2\in C(\left[ 0, \tau \right]; L^p(\Omega))$, we can use the items $(i)$ and $(ii)$ of Theorem \ref{espfunc02} to obtain
\begin{multline*}
    \| T\varphi_1(t,\cdot) - T\varphi_2(t,\cdot) \|_{L^p(\Omega)} \\ \leq \displaystyle \| A \|_{\mathcal{L}(L^p(\Omega))}  \| \varphi_1 -\varphi_2 \|_{C([0,\tau];L^p(\Omega))}\int_{0} ^{t} \left\| \dfrac{(t - s)^{\alpha(\cdot) - 1}}{\Gamma(\alpha(\cdot))} \right\|_{L^{\infty}(\Omega)}  ds \\
     \leq \displaystyle \| A \|_{\mathcal{L}(L^p(\Omega))}  \| \varphi_1 -\varphi_2\|_{C([0,\tau];L^p(\Omega))} \int_{0} ^{t} \dfrac{w^{\alpha_0-1}}{\Gamma(\|\alpha\|_{L^\infty(\Omega)})}  dw \\
    \leq \displaystyle   \underbrace{\dfrac{\tau^{\alpha_0}\| A \|_{\mathcal{L}(L^p(\Omega))}  }{\alpha_0\Gamma(\|\alpha\|_{L^\infty(\Omega)})}}_{\leq1/2}\| \varphi_1 -\varphi_2\|_{C([0,\tau];L^p(\Omega))},
\end{multline*}
for every $t\in[0,\tau]$.

As a result, we can deduce that the operator defined by the equation \eqref{mainform01} is a contraction mapping and, consequently, possesses a unique fixed point $u_1 \in C(\left[ 0, \tau \right]; L^p(\Omega))$.

Now, if we consider the set
$$H_{2\tau}=\Big\{ \varphi \in C(\left[ 0, 2\tau \right]; L^p(\Omega)) : \varphi(t,x) = u_1(t,x),\forall t\in[0,\tau]\textrm{ and a.e. }x\in\Omega\Big\},$$
Proposition \ref{auxlemmaprev} allows us to consider $T: H_{2\tau} \longrightarrow C(\left[ 0, 2\tau \right]; L^p(\Omega))$ with the same formula given in \eqref{mainform01}. Furthermore, note that for every $\varphi\in H_{2\tau}$, we have that
$$T\varphi(t,x)=Tu_1(t,x)=u_1(t,x)$$
for every $t\in[0,\tau]$ and a.e. $x\in\Omega$. So, we have $T: H_{2\tau} \longrightarrow H_{2\tau}$.

Using a similar argument as before, we can prove that for $\varphi_1, \varphi_2\in H_{2\tau}$, we have
$$\| T\varphi_1(t,\cdot) - T\varphi_2(t,\cdot) \|_{L^p(\Omega)}  \leq \displaystyle   \underbrace{\dfrac{\tau^{\alpha_0}\| A \|_{\mathcal{L}(L^p(\Omega))}  }{\alpha_0\Gamma(\|\alpha\|_{L^\infty(\Omega)})}}_{\leq1/2}\| \varphi_1 -\varphi_2\|_{C([0,2\tau];L^p(\Omega))},$$
for every $t\in[0,2\tau]$.

Therefore, we conclude that $T: H_{2\tau} \longrightarrow H_{2\tau}$ is a contraction mapping, and it has a unique fixed point $u_2\in C(\left[ 0, 2\tau \right]; L^p(\Omega))$ satisfying $u_2(t,x)=u_1(t,x)$ for every $t\in[0,\tau]$ and almost every $x\in\Omega$.

Repeating this process and applying Proposition \ref{integralequ}, we can establish the existence of a unique solution $u\in C([0,\infty);L^p(\Omega))$ to problem \eqref{eq1011} in $[0,\infty)\times\Omega$, as desired.
\end{proof}

\begin{remark} We emphasize that, in our view, the proof of Theorem \ref{exisuni} is already in its simplest form. A more standard approach would involve attempting to prove that a power of $T$ is a contraction, rather than restricting ourselves to a specific choice of $\tau \in (0,1/2)$. However, this method would have posed significant difficulties in deriving a formula to estimate the constant necessary for defining the contraction, making the proof of the theorem extremely difficult to complete.
\end{remark}

Our objective now is to better understand the solution obtained in Theorem \ref{exisuni} for problem \eqref{eq1011}. It is worth recalling that when $\alpha(x)$ is constant (i.e. $\alpha(x) = \alpha \in (0,1)$), the solution to the abstract fractional Cauchy problem
  \begin{align*}&cD_t^{\alpha} u(t,x)=Au(t,x),&& \forall t>0\textrm{ and a.e. }x\in\Omega,\\
&u(0,x)=u_0(x),&& \textrm{a.e. }x\in\Omega,
 \end{align*}
can be expressed in terms of the Mittag-Leffler function as
\begin{equation}\label{classicalmittag}u(t,x;u_0(x))=E_{\alpha}(t^\alpha A)u_0 (x):= \displaystyle \sum_{k = 0}^{\infty} \dfrac{t^{\alpha k} A^k u_0(x)}{\Gamma(\alpha k + 1)},\end{equation}
for every $t\in[0,\infty)$ and almost every $x\in\Omega$. Based on this formula, it is reasonable to assume that the unique solution to problem \eqref{eq1011} could be given by
\begin{equation}\label{novaserie1}
E_{\alpha(x)}\big(t^{\alpha(x)}A\big)u_0(x) := \sum_{k=0}^{\infty} \frac{t^{\alpha(x)k}A^ku_0(x)}{\Gamma(\alpha(x)k + 1)},
\end{equation}
for every $t\geq0$ and almost every $x\in\Omega$. With that in mind, before attempting any proofs, we should first ensure that the series above converges.

\begin{theorem}\label{convergentgeneral} The series defined in \eqref{novaserie1} is convergent in $L^p(\Omega)$ for every $t\geq0$. Moreover, there exists $M>0$ such that
\begin{multline}\label{202502171549}\big\|E_{\alpha(\cdot)}\big(t^{\alpha(\cdot)}A\big)u_0(\cdot)\big\|_{L^p(\Omega)}\\\leq \left\{\begin{array}{ll}Me^{\|A\|^{1/\alpha_0}_{\mathcal{L}(L^p(\Omega))}t}\| {u_0} \|_{L^p(\Omega)},&\textrm{if }t\in[0,1],\vspace*{0.2cm}\\
Me^{\|A\|^{1/\alpha_0}_{\mathcal{L}(L^p(\Omega))}t^{\|\alpha\|_{L^\infty(\Omega)}/\alpha_0}}\| {u_0} \|_{L^p(\Omega)},&\textrm{if }t\in(1,\infty).\end{array}\right.\end{multline}
for every $t\geq0$.
\end{theorem}
\begin{proof}Since $1/\Gamma(z)$ is decreasing on $(\eta, \infty)$ (recall Theorem \ref{espfunc02}), there exists $k_0^* \in \mathbb{N}$ such that
\begin{equation*}
    \dfrac{1}{\Gamma(\alpha(x)n + 1)} \leq \dfrac{1}{\Gamma(\alpha_0 n + 1)}, \quad \forall n \geq k_0^*, \quad \textrm{a.e.} \hspace{0.2cm} x \in \Omega.
\end{equation*}

On the other hande, Hölder's inequality guarantees
$$\left\|\dfrac{t^{\alpha(\cdot)n}A^n {u_0}(\cdot)}{\Gamma(\alpha(\cdot)n + 1)} \right\|_{L^{p}(\Omega)} \leq \left\| \dfrac{t^{\alpha(\cdot)n}}{\Gamma(\alpha(\cdot)n + 1)} \right\|_{L^{\infty}(\Omega)}\left\|A\right\|_{\mathcal{L}(L^p(\Omega))}^n \left\|{u_0}\right\|_{L^{p}(\Omega)}, $$
for every $t\geq0$. Now, let us split the remainder of the proof into two cases.

\underline{Case 1:} For $0 < t \leq 1$, items $(i)$ and $(ii)$ from Theorem \ref{espfunc02} ensure that
\begin{multline*}
    \left\| \sum_{n = 0}^{\infty} \dfrac{t^{\alpha(x)n}A^n {u_0}(x)}{\Gamma(\alpha(x)n + 1)} \right\|_{L^{p}(\Omega)} \\
     \leq \left[\sum_{n = 0}^{k_0^*} \dfrac{t^{\alpha_0 n}\| A \|_{\mathcal{L}(L^p(\Omega))}^n }{\Gamma(\eta)}\right]\| {u_0} \|_{L^p(\Omega)} + \left[\sum_{n =k_0^*}^{\infty} \dfrac{t^{\alpha_0 n}\| A \|_{\mathcal{L}(L^p(\Omega))}^n}{\Gamma(\alpha_0n + 1)}\right] \| {u_0} \|_{L^p(\Omega)}.
\end{multline*}
Therefore, we may conclude 
\begin{multline*}
    \left\| \sum_{n = 0}^{\infty} \dfrac{t^{\alpha(x)n}A^n {u_0}(x)}{\Gamma(\alpha(x)n + 1)} \right\|_{L^{p}(\Omega)} \leq \left[\sum_{n = 0}^{k_0^*} \dfrac{t^{\alpha_0 n}\| A \|_{\mathcal{L}(L^p(\Omega))}^n }{\Gamma(\eta)}\right]\| {u_0} \|_{L^p(\Omega)}\\ + E_{\alpha_0}(t^{\alpha_0} \| A \|_{\mathcal{L}(L^p(\Omega))})\| {u_0} \|_{L^p(\Omega)}<\infty,
\end{multline*}
for every $0 < t \leq 1$.

\underline{Case 2:} For $t > 1$, items $(i)$ and $(iii)$ from Theorem \ref{espfunc02}, along with a similar argument as above, justify that
\begin{multline*}
    \left\| \sum_{n = 0}^{\infty} \dfrac{t^{\alpha(x)n}A^n {u_0}(x)}{\Gamma(\alpha(x)n + 1)} \right\|_{L^{p}(\Omega)} \leq \left[\sum_{n = 0}^{k_0^*} \dfrac{t^{\|\alpha\|_{L^{\infty}(\Omega)}n}\| A \|_{\mathcal{L}(L^p(\Omega))}^n }{\Gamma(\eta)}\right]\| {u_0} \|_{L^p(\Omega)} \\
   + E_{\alpha_0}(t^{\|\alpha\|_{L^{\infty}(\Omega)}}\| A \|_{\mathcal{L}(L^p(\Omega))})\| {u_0} \|_{L^p(\Omega)}<\infty.
\end{multline*}

Now recall that \cite[Proposition 3.6]{gokima} ensures there exists $\tilde{M}>0$ such that for any $\beta\in(0,1]$ and $t\geq0$, it holds that
\begin{equation}\label{expmittag}E_\beta(t)\leq \tilde{M}e^{t^{1/\beta}}.\end{equation}
This, along with the conclusions of the Cases 1 and 2, shows that there exists $M\geq0$ such that \eqref{202502171549} holds.
\end{proof}

Even though Theorem \ref{convergentgeneral} ensures that the series defined in \eqref{novaserie1} is always convergent for \(t \geq 0\), it does not always represent the unique solution to problem \eqref{eq1011}, which is guaranteed to exist by Theorem \ref{exisuni}. Below, we present a theorem that provides sufficient conditions for this to be true.

\begin{theorem} \label{commute theorem}If for each $\phi\in L^p(\Omega)$, it holds that
\begin{equation}\label{comentrei}\left(\frac{s^{\alpha(x)}}{\Gamma(\alpha(x) + 1)}\right)A\phi(x)=A\left(\frac{s^{\alpha(x)}}{\Gamma(\alpha(x) + 1)}\phi(x)\right),\end{equation}
for every $s\geq0$ and almost every $x\in\Omega$, then the solution of problem \eqref{eq1011} is given by \eqref{novaserie1}.
\end{theorem}
\begin{proof}
Let $U_0(t, x) = u_0(x)$ and, for $n \geq 1$, define functions
$$U_{n}(t, x) = u_0(x)+\displaystyle \int_0^{t} \dfrac{(t-s)^{\alpha(x) - 1}}{\Gamma(\alpha(x))} AU_{n-1}(s,x) ds.$$
Note that, due to Theorem \ref{exisuni} and the uniqueness of the fixed point for the function $T: C(\left[ 0, k\tau \right]; L^p(\Omega)) \longrightarrow C(\left[ 0, k\tau \right]; L^p(\Omega))$ defined in \eqref{mainform01}, for any $k\in\mathbb{N}$, the unique solution $u(t,x)$ of problem \eqref{eq1011} satisfies
$$u(t,x) = \lim_{n\rightarrow\infty}U_n(t,x),$$
in the topology of $L^p(\Omega)$, uniformly on compact subsets of $[0,\infty)$.

However, by considering the fundamental identity \eqref{comentrei} and performing straightforward algebraic manipulations, we can establish the identity
$$U_n(t,x) = \displaystyle \sum_{k = 0}^{n} \dfrac{t^{k\alpha(x)}A^k u_0(x)}{\Gamma(k\alpha(x) + 1)},$$
for every $n\in\mathbb{N}$, as is standard in classical theory. Therefore
$$
		u(t, x) = \displaystyle \displaystyle \sum_{k = 0}^{\infty} \dfrac{t^{k\alpha(x)}A^k u_0(x)}{\Gamma(k\alpha(x) + 1)}= E_{\alpha(x)}(t^{\alpha(x)}A)u_0(x).
$$
\end{proof}

Below, we provide an example to illustrate a situation where we can apply Theorem \ref{commute theorem} to obtain the representation \eqref{novaserie1} for our solution.

\begin{example}
Let $\psi \in L^{\infty}(\Omega)$ be fixed and define $A: L^p(\Omega) \longrightarrow L^p(\Omega)$ by $A(\phi(x)) = \psi(x) \phi(x).$ Notice that for every $s \geq 0$ we have
\begin{multline*}
	A\left(\dfrac{s^{\alpha(x)}}{\Gamma(\alpha(x) + 1)}  \phi(x) \right)  = \psi(x) \left(\dfrac{s^{\alpha(x)}}{\Gamma(\alpha(x) + 1)} \phi(x)\right) = \\
	= \left(\dfrac{s^{\alpha(x)}}{\Gamma(\alpha(x) + 1)}\right) \psi(x) \phi(x)
	 = \left(\dfrac{s^{\alpha(x)}}{\Gamma(\alpha(x) + 1)}\right) A(\phi(x)),
\end{multline*}
for almost every $x\in\Omega$. Therefore, Theorem \ref{commute theorem} ensures that the solution of problem \eqref{eq1011} satisfies
$$u(t,x)=E_{\alpha(x)}(t^{\alpha(x)}A)(u_0(x))=\sum_{k = 0}^{\infty} \dfrac{t^{k\alpha(x)}\psi^k(x) u_0(x)}{\Gamma(k\alpha(x) + 1)},$$
for every $t\geq0$ and almost every $x\in\Omega$.
\end{example}

Finally, we present an example in which we cannot apply Theorem \ref{commute theorem}, making its solution quite complex. In this case, we cannot straightforwardly express the series that would describe the formation law of our solution.

\begin{example} \label{example integral}
Let $A: L^p(0,1) \longrightarrow L^p(0,1)$ defined by
$$A \phi(x) = \displaystyle \int _{0}^{x} \phi (y) dy.$$

Clearly, $A$ does not satisfy \eqref{comentrei} in general, preventing a straightforward application of Theorem \ref{commute theorem}. Furthermore, it is not trivial to establish a well-understood formula for the solution to this problem. With that in mind, to facilitate our discussion, let us assume that
\begin{equation*}
\alpha(x) = \left\{\begin{array}{ll}
\alpha_1 , &\textrm{if } x \in(0, 1/2] \vspace*{0.2cm}, \\
\alpha_2 , &\textrm{if } x \in (1/2, 1),\end{array}\right.
\end{equation*}
for fixed, positive, and non-zero values of $\alpha_1$ and $\alpha_2$ in $(0,1]$. In this way, the problem can be split into two parts.\vspace*{0.2cm}

\underline{Case 1:} Given that the solution $u(t,x)$ already exists (see the conclusions of Theorem \ref{exisuni}), by restricting ourselves to values of $x \in (0, {1/2}]$, the problem can be reinterpreted as
\begin{equation}\label{202502052353}
\begin{array}{ll}
cD_t^{\alpha_1} u(t,x) = A_1 u(t,x), \quad \forall t > 0\textrm{ and a.e. }x \in (0, {1/2}],  \vspace*{0.2cm}\\
u(0,x) = u_0(x), \textrm{ a.e. }x \in (0, {1/2}], \end{array}
\end{equation}
with $A_1: L^p(0,1/2) \longrightarrow L^p(0,1/2)$ being defined by
$$A_1 \phi(x) = \displaystyle \int _{0}^{x} \phi (y) dy.$$

Since $A_1 \in \mathcal{L}(L^p(0,1/2))$, we can apply the classical theory of fractional abstract differential equations to \eqref{202502052353} and obtain
\begin{equation} \label{solucao 1 - exemplo integral}
u(t,x)=E_{\alpha_1}(t^{\alpha_1}A_1)u_0(x),
\end{equation}
for every $t\geq0$ and almost every $x \in (0, {1/2}]$.\vspace*{0.2cm}

\underline{Case 2:} Now that we have the formula for the solution $u(t,x)$ when $x\in(0,1/2]$, for $x\in({1/2}, 1)$, we can reformulate the original problem as
\begin{equation*}
 \begin{array}{ll}
cD_t^{\alpha_2} u(t,x) = \displaystyle \int_0^{{1/2}} E_{\alpha_1}(t^{\alpha_1}A_1)u_0(y) dy + \displaystyle A_2 u(t,x), \\
u(0,x) = u_0(x),
\end{array}
\end{equation*}
 for all $t > 0$ and a.e. $x \in ({1/2} , 1)$, where $A_2: L^{p}({1/2},1) \longrightarrow L^{p}({1/2},1)$ is given by
\begin{equation*}
A_2 \phi(x) := \displaystyle \int_{{1/2}}^{x} \phi(y) dy.
\end{equation*}

Again, since $A_2\in\mathcal{L}(L^p(1/2,1))$, we can use the fractional variation of constants formula to obtain that
\begin{multline} \label{solucao 2 - exemplo integral}
u(t, x) = E_{\alpha_2}(t^{\alpha_2}A_2)u_0(x) \\+ \displaystyle \int_0^t (t - s)^{\alpha_2 - 1} E_{\alpha_2, \alpha_2} \big((t - s)^{\alpha_2} A_2\big)\left[\int_0^{{1/2}} E_{\alpha_1}(s^{\alpha_1}A_1)u_0(y) dy\right]ds,
\end{multline}
for every $t\geq0$ and almost every $x \in (1/2,1)$.

Therefore, the solution to problem \eqref{eq1011} for the given operator $A$ is defined in parts, with the expressions \eqref{solucao 1 - exemplo integral} for $x \in (0, {1/2}]$ and \eqref{solucao 2 - exemplo integral} for $x \in ({1/2}, 1)$.
\end{example}

To conclude this section, we present an important discussion about the boundedness of the family of operators that defines the solution of \eqref{eq1011}. Recall that for any $\phi\in L^p(\Omega)$, Theorem \ref{exisuni} guarantees the uniqueness and existence of the solution to the Cauchy problem \eqref{eq1011}, which we denote here by $u(t,x;\phi)$. Consequently, we can define a family of bounded operators $\{S(t):t\geq0\}\subset\mathcal{L}(L^p(\Omega))$, given by 
\begin{equation}\label{matrixoperator}S(t)\phi(x):=u(t,x;\phi).\end{equation}

In the case where $\alpha(x)=\alpha$ is constant in $\Omega$, by recalling the formula in \eqref{classicalmittag} and the estimate \eqref{expmittag}, we can prove that
$$\|S(t)\|_{\mathcal{L}(L^p(\Omega))}=\|E_\alpha(t^\alpha A)\|_{\mathcal{L}(L^p(\Omega))}\leq E_\alpha(t^\alpha \|A\|_{\mathcal{L}(L^p(\Omega))})\leq\tilde{M}e^{\|A\|^{1/\alpha}_{\mathcal{L}(L^p(\Omega))}t},$$
for every $t\geq0$. In other words, we have demonstrated that the operator is exponentially bounded. This behavior plays a crucial role in computing the Laplace transform of the operator, as it provides valuable insights. Hence, it is essential to investigate whether a similar exponential boundedness can be achieved when $\alpha(x)$ varies with $x$.

To achieve this, first observe that if \eqref{comentrei} holds, then we can conclude that 
$$
S(t) = E_{\alpha(x)}\big(t^{\alpha(x)}A\big)u_0(x).
$$

Recalling the steps from the proof of Theorem \ref{convergentgeneral}, it becomes clear that proving the family $\{S(t):t\geq0\}$ is exponentially bounded is not straightforward. For example, when $t > 1$, the prove only gives us
$$
\|S(t)\|_{\mathcal{L}(L^p(\Omega))} \leq M e^{\|A\|_{\mathcal{L}(L^p(\Omega))}t^{\|\alpha\|_\infty/\alpha_0}},
$$
where $\|\alpha\|_\infty/\alpha_0 > 1$. However, in certain cases, exponential boundedness can be directly proven. This is the main topic addressed in the next result, which is also the last result of this section.

\begin{theorem} 
Assume that \eqref{comentrei} holds. If there exist subsets $\{\Omega_j\}_{j=1}^m \subset \Omega$ such that $\cup_{j=1}^m \Omega_j = \Omega$ and $\{\alpha_j\}_{j=1}^m \subset (0,1]$, with $\alpha(x) = \alpha_j$ for $x \in \Omega_j$, then there exist positive constants $\Lambda$ and $M$ such that the family of operators $\{S(t) : t \geq 0\}$, defined in \eqref{matrixoperator}, satisfies the inequality
$$
\|S(t)\|_{\mathcal{L}(L^p(\Omega))} \leq M e^{\Lambda t},
$$
for every $t \geq 0$.
\end{theorem}

\begin{proof} 
Let $t>0$ and $k\in\mathbb{N}$. It is not difficult to notice that 
$$
\left\|\dfrac{t^{\alpha(x)k}}{\Gamma(\alpha(x)k+1)}\right\|_{L^\infty(\Omega)}\leq\sum_{j=1}^m\left[\dfrac{t^{\alpha_jk}}{\Gamma(\alpha_jk+1)}\right],
$$
for almost every $x\in\Omega$. Since Theorem \ref{commute theorem} ensures that 
$$
S(t)u_0(x) = E_{\alpha(x)}\big(t^{\alpha(x)}A\big)u_0(x),
$$ 
we have
\begin{align*}
\left\|S(t)u_0\right\|_{L^p(\Omega)} & \leq\sum_{k=0}^\infty\left\|\dfrac{t^{\alpha(\cdot)k}}{\Gamma(\alpha(\cdot)k+1)}\right\|_{L^\infty(\Omega)}
\|A\|_{\mathcal{L}(L^p(\Omega))}^k\|u_0\|_{L^p(\Omega)} \\
& \leq\sum_{k=0}^\infty\sum_{j=1}^m\left[\dfrac{t^{\alpha_jk}}{\Gamma(\alpha_jk+1)}\right]
\|A\|_{\mathcal{L}(L^p(\Omega))}^k\|u_0\|_{L^p(\Omega)},
\end{align*}
for every $t\geq0$, which directly implies that
$$
\left\|S(t)\right\|_{\mathcal{L}(L^p(\Omega))}\leq\sum_{j=1}^m E_{\alpha_j}(t^{\alpha_j} \|A\|_{\mathcal{L}(L^p(\Omega))}),
$$
for every $t\geq0$. Finally, by recalling \eqref{expmittag} we achieve the boundedness
$$
\left\|S(t)\right\|_{\mathcal{L}(L^p(\Omega))}\leq\sum_{j=1}^m M_je^{t\|A\|^{1/\alpha_j}_{\mathcal{L}(L^p(\Omega))}},
$$
for every $t\geq0$. From the above inequality, simply set 
\begin{multline*}
    M=m\max\big\{M_j:j\in\{1,\ldots,m\}\big\} \\ 
     \textrm{and} \quad \Lambda=\max\big\{\|A\|_{\mathcal{L}(L^p(\Omega))}^{1/\alpha_j}:j\in\{1,\ldots,m\}\big\},
\end{multline*}
to finally obtain
$$
\left\|S(t)\right\|_{\mathcal{L}(L^p(\Omega))}\leq Me^{\Lambda t},
$$
for every $t\geq0$.
\end{proof}

\section{Well-Posedness of the Semilinear Problem}
\label{last}

In addition to the hypotheses stated at the beginning of Section \ref{linearprob}, we introduce a function $f:[0, \infty)\times L^p(\Omega)\rightarrow L^p(\Omega)$ to formulate the semilinear fractional Cauchy problem for the space-dependent fractional evolution equation of order $\alpha(x)$, given by the equations
\begin{subequations}\label{eqsemilinear01}
  \begin{align}\label{eqsemilinear011}&cD_t^{\alpha(x)} u(t,x) = Au(t,x) + f(t, u(t, x)),&& \forall t>0\textrm{ and a.e. }x\in\Omega,\\
\label{eqsemilinear012}&u(0,x)=u_0(x),&& \textrm{ a.e. }x\in\Omega.
 \end{align}
\end{subequations}
Observe that the subsequent steps in this section would involve defining the notion of a solution and establishing the relationship between the solution of \eqref{eqsemilinear01} and the corresponding integral equation, much like what was achieved with Definition \ref{integralequ1} and Proposition \ref{integralequ} in the linear case. To avoid unnecessary repetition of proofs and arguments, while maintaining clarity for the reader, we skip certain proofs that follow the same reasoning already presented in Section \ref{linearprob}. This approach helps keep the section concise.

\begin{definition} We say that $\varphi\in C(I;L^p(\Omega))$ is solution of \eqref{eqsemilinear01} in $I\times\Omega$ if it satisfies:
\begin{itemize}
\item[(i)] $cD_t^{\alpha(x)} u(t,x) = Au(t,x) + f(t, u(t, x))$, for every $t\in I$ and almost every $x\in\Omega.$\vspace*{0.2cm}
\item[(ii)] $u(0,x)=u_0(x)$, for almost every $x\in\Omega.$
\end{itemize}
\end{definition}

Now we present a general version of Proposition \ref{integralequ}.
\begin{proposition}\label{intsemilinear}
Let $T>0$, $f \in C([0, \infty)\times L^p(\Omega);L^p(\Omega))$ and $\varphi$ a function such that $\varphi \in C([0,T];L^p(\Omega))$. Then $\varphi(t,x)$ is a solution of \eqref{eqsemilinear01} in $[0,T]\times\Omega$ if, and only if, $\varphi(t,x)$ satisfies the integral equation
\begin{multline*}
\varphi(t,x) = u_0(x) + \dfrac{1}{\Gamma(\alpha(x))}\int_{0}^t{(t-s)^{\alpha(x)-1}A\varphi(s,x)}\,ds \\
	+ \dfrac{1}{\Gamma(\alpha(x))}\int_{0}^t{(t-s)^{\alpha(x)-1}f(s,x, \varphi(s,x))}\,ds,
\end{multline*}
for every $t \in [0,T]$ and almost every $x \in \Omega$.
\end{proposition}
\begin{proof}It is important to emphasize that the hypotheses on $f:[0, \infty) \times L^p(\Omega) \to L^p(\Omega)$ are sufficient for us to adapt the proof of Proposition \ref{integralequ} in order to establish this result. Therefore, as mentioned earlier, we have opted to omit these redundant arguments.
\end{proof}

To proceed, we introduce the definition of the locally Lipschitz notion considered in this work.
\begin{definition} \label{locallyLipschitz}
We say that $f:[0, \infty) \times L^p(\Omega)\rightarrow L^p(\Omega)$ is locally Lipschitz in the second variable, uniformly with respect to the first variable, if for each fixed $(t_0, \phi_0) \in [0, \infty) \times L^p(\Omega)$, there exist $r_0, L_0 > 0$ such that
\begin{equation*}
\| f (t, \psi)- f (t, \phi) \|_{L^p(\Omega)} \leq L_0 \| \psi-\phi \|_{L^p(\Omega)}, 
\end{equation*}
for every $(t, \psi),(t, \phi)\in B_{r_0}(t_0,\phi_0)$, where
$$B_{r_0}(t_0,\phi_0):=\big\{(s,\lambda)\in [0,\infty)\times L^p(\Omega):|s-t_0|+\|\lambda-\phi_0\|_{L^p(\Omega)}<r_0\big\}.$$
\end{definition}

Based on the notions and results established so far in this paper, we now present one of the main results of this section.
\begin{theorem}\label{semilinearexisuni} If $f \in C([0, \infty)\times L^p(\Omega);L^p(\Omega))$ is locally Lipschitz in the second variable, uniformly with respect to the first variable, then there exists $\tau_0>0$ such that \eqref{eqsemilinear01} has a unique solution in $[0,\tau_0]\times\Omega$.
\end{theorem}
\begin{proof}
For the initial condition $(0,u_0)\in[0,\infty)\times L^p(\Omega)$, assume that $r_0,L_0 > 0$ and $B_{r_0}(0,u_0)$ are giving according to Definition \ref{locallyLipschitz}. Set the parameters
\begin{multline*}\beta\in(0,r_0),\quad M_0=\max\big\{\|f(s,u_0(\cdot))\|_{L^p(\Omega)}:s\in[0,r_0]\big\} \quad \textrm{and} \\
\tau_0 = \min \left\{r_0-\beta, \left[\dfrac{\alpha_0\Gamma(\|\alpha\|_{L^\infty(\Omega)})}{2\big(\| A \|_{\mathcal{L}(L^p(\Omega))}+L_0\big)}\right]^{1/\alpha_0}, \left[\dfrac{\beta\alpha_0\Gamma(\|\alpha\|_{L^\infty(\Omega)})}{2M_0+1}\right]^{1/\alpha_0}\right\}.\end{multline*}
Now, consider the set 
\begin{multline*}K_0:=\Big\{\varphi\in C([0, \tau_0];L^p(\Omega)):\varphi(0,x)=u_0(x)\textrm{ for a.e. }x\in\Omega\\\textrm{ and }\|\varphi(t,\cdot)-u_0(\cdot)\|_{L^p(\Omega)}\leq\beta,\textrm{ for every }t\in[0,\tau_0]\Big\},\end{multline*}
and $T:K_0\rightarrow C([0, \tau_0];L^p(\Omega))$ given by
\begin{multline*}
T\varphi(t,x) = u_0(x) + \dfrac{1}{\Gamma(\alpha(x))}\int_{0}^t{(t-s)^{\alpha(x)-1}A\varphi(s,x)}\,ds \\
	+ \dfrac{1}{\Gamma(\alpha(x))}\int_{0}^t{(t-s)^{\alpha(x)-1}f(s, \varphi(s,x))}\,ds.
\end{multline*}

Let us prove that $T(K_0)\subset K_0$. First, observe that if $\varphi\in K_0$, then, by Proposition \ref{auxlemmaprev}, we have $T\varphi(0,x)=u_0(x)$ for almost every $x\in\Omega$, and that $T\varphi\in C([0,\tau_0];L^p(\Omega))$. 

Moreover, since both $(s,u_0(x))$ and $(s,\varphi(s,x))$ belong to $B_{r_0}(0,u_0)$ for all $s\in[0,\tau_0]$ and almost every $x\in\Omega$, applying the estimates established in the proof of Theorem \ref{exisuni}, we obtain
\begin{multline*}
    \| T\varphi(t,\cdot) - u_0(\cdot) \|_{L^p(\Omega)} \leq \displaystyle   \left[{\dfrac{\tau_0^{\alpha_0}\Big(\| A \|_{\mathcal{L}(L^p(\Omega))}+L_0\Big)  }{\alpha_0\Gamma(\|\alpha\|_{L^\infty(\Omega)})}}\right]
    \|\varphi - u_0\|_{C([0,\tau];L^p(\Omega))}\\+\left[{\dfrac{\tau_0^{\alpha_0}M_0  }{\alpha_0\Gamma(\|\alpha\|_{L^\infty(\Omega)})}}\right]\leq\beta,
\end{multline*}
for every $t\in[0,\tau_0]$, as we wanted. Now, if $\varphi_1,\varphi_2\in K_0$, both $(s,\varphi_1(s,x))$ and $(s,\varphi_2(s,x))$ belong to $B_{r_0}(0,u_0)$ for all $s\in[0,\tau_0]$ and almost every $x\in\Omega$, therefore
\begin{equation*}
    \| T\varphi_1(t,\cdot) -T\varphi_2(t,\cdot) \|_{L^p(\Omega)} \leq \displaystyle   \underbrace{\left[{\dfrac{\tau_0^{\alpha_0}\Big(\| A \|_{\mathcal{L}(L^p(\Omega))}+L_0\Big)  }{\alpha_0\Gamma(\|\alpha\|_{L^\infty(\Omega)})}}\right]}_{\leq1/2}
    \|\varphi - u_0\|_{C([0,\tau];L^p(\Omega))},
\end{equation*}
for every $t\in[0,\tau_0]$.

Thus, by the Banach contraction principle, we conclude that $T$ has a unique fixed point in $K_0$, which completes the proof of the theorem.
\end{proof}

From this point onward, we focus on presenting results concerning extensions of the solution whose existence was established by the theorem above. Additionally, we aim to justify its maximal existence, in the sense that it defines the largest interval in which the solution remains valid. To this end, we first introduce some necessary concepts.

\begin{definition}
\begin{itemize}
\item[(i)] Let $\varphi:[0,{\tau}]\rightarrow L^p(\Omega)$ be the unique solution to \eqref{eqsemilinear01} on $[0,{\tau}]\times\Omega$. If ${\tau^*} > {\tau}$ and $\varphi^*:[0,{\tau^*}]\rightarrow L^p(\Omega)$ is a solution to \eqref{eqsemilinear01} on $[0,{\tau^*}]\times\Omega$, then we say that $\varphi^*$ is a continuation of $\varphi$ over $[0,{\tau^*}]\times\Omega$.
\item[(ii)] If $\varphi:[0,{{\tau^*}})\rightarrow L^p(\Omega)$ is the unique solution to \eqref{eqsemilinear01} on $[0,{\tau}]\times\Omega$ for all $\tau \in (0,\tau^*)$ and cannot be extended to $[0,{\tau}^*]\times\Omega$ while remaining a solution, then we call $\varphi$ a maximal solution to \eqref{eqsemilinear01} in $[0,\tau^*)\times \Omega$.
\end{itemize}
\end{definition}

\begin{lemma} \label{lemablowupsemilinear}
Let $f \in C([0, \infty) \times L^p(\Omega); L^p(\Omega))$ be locally Lipschitz in the second variable, uniformly with respect to the first variable. If $\varphi \in C([0, \tau]; L^p(\Omega))$ is a solution of \eqref{eqsemilinear01} on $[0, \tau]\times\Omega$, then there exists a unique continuation $\varphi^*$ of $\varphi$ to an extended interval $[0, \tau + \tau^*]\times\Omega$ for some $\tau^* > 0$.
\end{lemma}
\begin{proof}
Let $\varphi$ be the solution to \eqref{eqsemilinear01} in $[0, \tau]\times\Omega$. Since $f$ is locally Lipschitz in the second variable, uniformly with respect to the first variable, for $(\tau,\varphi(\tau))\in[0,\infty)\times L^p(\Omega)$, assume that $r_\tau,L_\tau > 0$ and $B_{r}(\tau,\varphi(\tau))$ are giving according to Definition \ref{locallyLipschitz}. Set the parameters
\begin{multline*}\beta\in(0,r_\tau),\quad M_\tau=\max\big\{\|f(s,\varphi(\tau,\cdot))\|_{L^p(\Omega)}:s\in[0,r_\tau]\big\} \quad \textrm{and} \\ \tau^* = \min \left\{r_\tau-\beta, \left[\dfrac{\alpha_0\Gamma(\|\alpha\|_{L^\infty(\Omega)})}{2\big(\| A \|_{\mathcal{L}(L^p(\Omega))}+L_\tau\big)}\right]^{1/\alpha_0}, \left[\dfrac{\beta\alpha_0\Gamma(\|\alpha\|_{L^\infty(\Omega)})}{2L_\tau M_\tau+1\big)}\right]^{1/\alpha_0}\right\}.\end{multline*}
Now, consider the set 
\begin{multline*}K_\tau:=\Big\{\psi\in C([0, \tau+\tau^*];L^p(\Omega)):\psi(s,x)=\varphi(s,x)\textrm{ for every }t\in[0,\tau] \\
\textrm{ and a.e. }x\in\Omega \textrm{ and }\|\psi(t,\cdot)-\varphi(\tau,\cdot)\|_{L^p(\Omega)}\leq\beta,\textrm{ for every }t\in[\tau,\tau+\tau^*]\Big\},\end{multline*}
and $T:K_\tau\rightarrow C([0, \tau];L^p(\Omega))$ given by
\begin{multline*}
T\psi(t,x) = u_0(x) + \dfrac{1}{\Gamma(\alpha(x))}\int_{0}^t{(t-s)^{\alpha(x)-1}A\psi(s,x)}\,ds \\
	+ \dfrac{1}{\Gamma(\alpha(x))}\int_{0}^t{(t-s)^{\alpha(x)-1}f(s, \psi(s,x))}\,ds.
\end{multline*}

Following the same approach as in Theorem \ref{semilinearexisuni}, we prove that $T(K_\tau) \subset K_\tau$ and that $T$ is a contraction. Therefore, by the Banach contraction principle, we conclude that there exists a unique solution $\varphi^*$ on $[0, \tau + \tau^*]\times\Omega$ to \eqref{eqsemilinear01}, which is the unique continuation of $\varphi$ over $[0,\tau^*]\times\Omega$.
\end{proof}

We now present our final result, which concludes the theory developed here on semilinear differential equations involving the Caputo fractional derivative of order $\alpha(x)$ and bounded operators $A \in \mathcal{L}(L^p(\Omega))$. This result is fundamental in addressing the blow-up of maximal solutions within bounded intervals of $[0,\infty)$ or the global existence of solutions.

\begin{theorem} \label{blowupsemilinear}
Assume that $f \in C([0, \infty) \times L^p(\Omega); L^p(\Omega))$ is  locally Lipschitz in the second variable, uniformly with respect to the first variable, and maps bounded sets onto bounded sets. Then, either problem \eqref{eqsemilinear01} admits a global solution $\varphi\in C([0, \infty);L^p(\Omega))$, or there exists $\omega \in (0, \infty)$ such that $\varphi \in C([0,\omega);L^p(\Omega))$ is a maximal solution of \eqref{eqsemilinear01} on $[0,\omega)\times\Omega$ satisfying
$$\limsup_{t \rightarrow \omega^{-}} \| \varphi(t,\cdot) \|_{L^p(\Omega)} = \infty.$$
\end{theorem}

\begin{proof}
According to Theorem \ref{semilinearexisuni}, the set
\begin{multline*}
    H:= \big\{ \tau \in [0, \infty) :\textrm{ exists }\varphi_{\tau} \in C(\left[ 0, \tau \right]; L^p(\Omega)) \\ \textrm{unique solution to }\eqref{eqsemilinear01}\textrm{ in }[0, \tau]\times\Omega \big\},
\end{multline*}
is not empty.

If we denote $\omega = \sup H$, we can define the continuous function $\varphi: [0, \omega) \rightarrow L^{p}(\Omega)$, which is also a solution to \eqref{eqsemilinear01} in $[0,\omega)\times\Omega$, by
$$\varphi(t,x)=\varphi_\tau(t,x),\textrm{ for }t\in[0,\tau]\textrm{ and a.e. }x\in\Omega.$$

\begin{itemize}
\item[(i)] If $\omega = \infty$, then $\varphi$ is a solution to \eqref{eqsemilinear01} on $[0, \infty)\times\Omega$, and there is nothing further to prove. 
\item[(ii)] If $\omega < \infty$, we shall need to prove that $\varphi$ does not admits a continuation and that 
\begin{equation}\label{202502101706}\limsup_{t \rightarrow \omega^{-}} \| \varphi(t,\cdot) \|_{L^p(\Omega)} = \infty.\end{equation}
\begin{itemize}
\item[(a)] Suppose, for the sake of contradiction, that $\varphi$ can be extended to $[0,\omega]\times\Omega$ while remaining a solution to \eqref{eqsemilinear01} on $[0,\omega]\times\Omega$. Then, by Lemma \ref{lemablowupsemilinear}, there exists $\tau^* > 0$ such that $\varphi^*:[0,\omega+\tau^*] \to L^p(\Omega)$ is the unique continuation of $\varphi$ on $[0,\omega+\tau^*]\times\Omega$. 

However, this implies that $\omega + \tau^* \in H$, leading to the contradiction $\omega + \tau^* \leq \omega$, or equivalently, $\tau^* \leq 0$, which is impossible. Therefore, $\varphi$ is a maximal solution to \eqref{eqsemilinear01} on $[0,\omega)\times\Omega$.

\item[(b)] Assume that \eqref{202502101706} does not hold. Then, there exists a constant $K > 0$ such that
$$\| \varphi(t,\cdot) \|_{L^p(\Omega)} \leq K, \quad \text{for all } t \in [0,\omega).$$
Since the set $\{(t, \varphi(t,\cdot)) : t \in [0,\omega)\}$ is bounded in $[0,\infty) \times L^p(\Omega)$, we define
$$S = \sup_{t \in [0, \omega)} \| f(t, \varphi(t, \cdot)) \|_{L^p(\Omega)}.$$

Now, consider a sequence $\{ t_n \}_{n \in \mathbb{N}} \subset [0, \omega)$ such that $t_n \to \omega$. By Theorem \ref{espfunc02}, and using estimates similar to those in Theorem \ref{semilinearexisuni}, it follows that $\{ \varphi(t_n,\cdot) \}_{n \in \mathbb{N}}$ is a Cauchy sequence in $L^p(\Omega)$. Consequently, $\varphi$ can be extended to $[0, \omega]$ in a well-defined manner, satisfying \eqref{intsemilinear} for all $t \in [0, \omega]$ and almost every $x \in \Omega$. Since this follows from a standard argument, we omit the details.

Furthermore, by Lemma \ref{blowupsemilinear}, the solution can be extended to a larger interval, contradicting the definition of $\omega$. Therefore, if $\omega < \infty$, the contradiction above implies that \eqref{202502101706} must hold, which completes the proof.

\end{itemize}
\end{itemize}
\end{proof}

%%%%%%%%%%%%%%%%%%%%%%%%%%%%%%%%%%%%%%%%%%%%%%%%%%%%%%%%%%%%%%%%%

\section{Application of the Theory}
\label{applications}
In this section, we outline the reasoning behind adapting certain mathematical models for analysis within the theoretical framework developed in this paper. As a main motivational example, we consider the SIR (Susceptible, Infected, Recovered) model, a widely used framework for describing disease transmission dynamics.

\subsection{The Classical SIR Model}
The SIR model provides a mathematical structure for analyzing how infection and recovery processes shape pathogen spread within a population, offering insights into the underlying mechanisms driving epidemic dynamics.

The model consists of three key compartments that describe the population dynamics:

\begin{itemize}
    \item[(i)] Susceptible, \( S(t) \): individuals who are at risk of infection;
    \item[(ii)] Infected, \( I(t) \): individuals who are infected and capable of transmitting the disease;
    \item[(iii)] Recovered, \( R(t) \): individuals who have recovered or been removed from the infectious population and are no longer susceptible.
\end{itemize}

The system of nonlinear differential equations governing the SIR model are
\begin{equation}\label{202502061253}\begin{array}{ll}
S^\prime(t) = -\beta S(t) I(t),\\
I^\prime(t) = \beta S(t) I(t) - \gamma I(t),\\
R^\prime(t) = \gamma I(t),
\end{array}
\end{equation}
where $\beta$ is the transmission rate, representing the probability that an interaction between a susceptible individual and an infected individual results in a new infection and $\gamma$ is the recovery rate, which determines the rate at which infected individuals recover and transition into the recovered compartment.

\subsection{Addressing Limitations of the Classical SIR Model}
It is important to note that the SIR model \eqref{202502061253} assumes that disease transmission occurs only through direct, local interactions between individuals. This simplification neglects the role of population movement across different regions, which is a crucial factor in real-world epidemic dynamics. As a result, the model may fail to capture the full complexity of disease spread, particularly in highly connected urban environments.

Understanding the environmental factors that influence the spread of infectious diseases is essential for developing effective prevention strategies, as highlighted by \cite{yin}. A key question posed by \cite{gerritse} is whether cities themselves accelerate COVID-19 transmission, or whether the movement of individuals between them plays a more significant role in the spread of pandemics.

Several studies provide insights into this issue. According to \cite{yin}, population density in urban areas of Hubei province and across the 51 states and territories of the USA correlates with higher morbidity rates, particularly in regions severely affected by COVID-19. However, \cite{gerritse} suggests that transmission dynamics differ significantly between urban centers and other locations, raising the question of whether connectivity, rather than pure population density, is the primary driver of disease dissemination. Some studies argue that densely populated areas facilitate disease transmission due to increased interpersonal contact, while others emphasize that human mobility and transport networks play a dominant role.

These observations indicate that a more realistic epidemiological model should incorporate both spatial interactions and human mobility. 

%%%%%%%%%%%%%%%%%%%%%%%%%%%%%%%%%%%%%%%%%%%%%%%%%%%%%%%%%%%%%%%%%

\subsection{A Fractional SIR Model for Epidemic Dynamics}

The classical SIR model assumes that individuals transition between compartments (Susceptible, Infected, and Recovered) at fixed rates, without accounting for memory effects in infection and recovery. However, real-world disease transmission often deviates from this Markovian assumption. As highlighted by \cite{angstmann2016fractional,angstmann2017}, the probability of remaining infected or recovering frequently depends on the elapsed time since infection, leading to more complex, time-dependent behaviors that cannot be captured by standard integer-order models.

A key motivation for incorporating fractional derivatives in epidemiological modeling arises from stochastic processes governing disease dynamics. \cite{angstmann2016fractional,angstmann2017} derive a fractional-order SIR model using a continuous-time random walk (CTRW) framework, demonstrating that when infectivity follows a power-law distribution, the resulting equations naturally take the form of fractional differential equations. This approach introduces a memory effect, meaning that infection and recovery processes depend not only on the present state but also on past states of the system, leading to a more realistic characterization of disease transmission.

Empirical studies on diseases such as COVID-19 and Ebola further support the necessity of fractional-order modeling, as infection dynamics vary based on how long an individual has been infected. Fractional derivatives provide a flexible mathematical framework for incorporating these effects. In particular, the Caputo fractional derivative is widely used in epidemiology due to its ability to retain historical infection data while preserving classical boundary conditions. This enables a more accurate representation of delayed transmission, variable immunity effects, and heterogeneous recovery rates.

%%%%%%%%%%%%%%%%%%%%%%%%%%%%%%%%%%%%%%%%%%%%%%%%%%%%%%%%%%%%%%%%%

\subsection{A Generalized SIR Model with Fractional Differentiation Depending on Spatial Position}

In contexts where mobility and connectivity are crucial, even fractional SIR models often fail to capture the full complexity of disease spread. While these models introduce memory effects in infection and recovery, they still assume local transmission, neglecting spatial interactions. However, real-world epidemics are strongly influenced by population mobility, where individuals interact across multiple regions, enabling non-local transmission. A more realistic epidemiological model should account for both non-local transmission and region-dependent infection dynamics.

To address this, we incorporate spatial interactions into the infection process. A natural approach is to introduce an integral term in the infection equation, allowing the infection rate at position $x$ to depend on the spatial distribution of infected individuals rather than solely local interactions. Recent models have employed fractional reaction–diffusion equations to more accurately represent non-local transmission mechanisms, capturing spatial heterogeneity and aligning with real-world disease spread \cite{lu2022}. In contrast, our approach introduces a distinct framework.

Mathematically, we introduce the following spatially dependent infection equation:
\begin{equation}\label{202502071508}
\frac{\partial I(t, x)}{\partial t} = \beta S(t, x) \int_0^x I(t, y) \, dy - \gamma I(t, x).
\end{equation}
Here, the integral term 
$$\int_0^x I(t, y) \, dy,$$
represents the cumulative number of infected individuals over the spatial domain. This formulation ensures that the local infection rate accounts for prior spatial locations where interactions have occurred, incorporating movement-based transmission effects. Such an approach is particularly relevant in urban environments, where movement between neighborhoods or transport hubs significantly affects disease spread.

While changing the second equation of \eqref{202502061253} to \eqref{202502071508} improves the model by introducing non-local disease transmission, it does not fully account for memory effects in different regions. In particular, infection and recovery rates may vary across space due to differences in population density, healthcare availability, and social behaviors. To address this limitation, we introduce a fractional time derivative with order depending on the spatial position $x$, capturing region-dependent memory effects.

The introduction of a space-dependent fractional derivative is justified by its capability to model heterogeneous infection dynamics across different spatial regions. As demonstrated in \cite{angstmann2016fractional}, fractional derivatives naturally appear in disease models governed by anomalous diffusion or power-law waiting times. Extending these models to spatially heterogeneous systems allows us to incorporate region-dependent memory effects, where past infections influence disease progression differently in each location.

With this in mind, we now propose the following generalized fractional SIR model:
\begin{equation}\label{202502071511}
\begin{array}{ll}
cD_t^{\alpha(x)}S(t, x) = -\beta S(t, x) I(t, x),\vspace*{0.2cm}\\
cD_t^{\alpha(x)} I(t, x) = \beta S(t, x) \displaystyle\int_0^x I(t, y) \, dy - \gamma I(t, x),\vspace*{0.2cm}\\
cD_t^{\alpha(x)}R(t, x) = \gamma I(t, x),
\end{array}
\end{equation}
where \( t\in(0,\infty) \), \( x\in[0,1] \) (representing the population's position), and \( cD_t^{\alpha(x)} \) is the Caputo fractional derivative of order \( \alpha(x) \) introduced in Section \ref{fracintro}.

Therefore, given an initial condition, Theorems \ref{semilinearexisuni} and \ref{blowupsemilinear} provide sufficient conditions to guarantee the existence and uniqueness of a solution to \eqref{202502071511}, while also ensuring its global existence. We emphasize that, to avoid unnecessary lengthening of the paper, we do not explicitly carry out these computations. However, these results follow from well-known arguments and can be readily obtained.

%%%%%%%%%%%%%%%%%%%%%%%%%%%%%%%%%%%%%%%%%%%%%%%%%%%%%%%%%%%%%%%%%

\subsection{Final Conclusions}

The proposed formulation \eqref{202502071511} introduces three key modifications to the classical SIR model:
\begin{itemize}
    \item[(i)] Spatially varying memory effects: Different regions exhibit distinct rates of infection progression. For instance, densely populated urban centers may experience accelerated transmission, whereas rural areas, with lower contact rates, may undergo delayed exposure.
    \item[(ii)] Non-local transmission dynamics: The integral term ensures that infection rates depend not only on local interactions but also on cumulative exposure across space, better capturing real-world epidemic patterns.
    \item[(iii)] Coupling spatial integration with fractional differentiation: This combination allows for a more realistic representation of epidemic dynamics, where disease transmission is influenced by both historical exposure and spatial movement \cite{angstmann2017,lu2022}. Additionally, incorporating a space-dependent fractional order \( \alpha(x) \) enhances the model’s ability to capture regional variations in infection and recovery rates, making it more accurate for heterogeneous populations.
\end{itemize}

This formulation also opens several new research directions, including:
\begin{itemize}
    \item[(i)] Analyzing the stability and asymptotic behavior of solutions under different assumptions for \( \alpha(x) \).
    \item[(ii)] Exploring optimal control strategies for epidemic mitigation in heterogeneous populations.
    \item[(iii)] Developing numerical schemes tailored to solve fractional differential equations with space-dependent derivatives.
    \item[(iv)] Validating the model against real epidemiological data to assess its predictive capabilities.
\end{itemize}

The integration of spatial heterogeneity and fractional differentiation in epidemiological modeling marks a significant step toward a more comprehensive understanding of disease propagation. By capturing memory effects and non-local transmission mechanisms, the proposed framework bridges the gap between classical compartmental models and complex real-world dynamics. Future investigations should focus on refining the model’s numerical implementation, analyzing its stability under various epidemiological scenarios, and validating its applicability through empirical data. As fractional modeling continues to gain traction in applied sciences, this approach offers a promising avenue for improving epidemic forecasting and intervention strategies.

%%%%%%%%%%%%%%%%%%%%%%%%%%%%%%%%%%


\begin{thebibliography}{00}


\bibitem{AbSt1} M. Abramowitz and I. A. Stegun. \textit{Handbook of Mathematical Functions, With Formulas, Graphs, and Mathematical Tables}, Dover Publications, New York, 1974.

\bibitem{AlKw1} H. Alzer and M. K. Kwong. On Ivády’s bounds for the gamma function and related results, \textit{Periodica Mathematica Hungarica}, vol. 82, pp. 115–124, 2021.

\bibitem{angstmann2016fractional} C. N. Angstmann, B. I. Henry, A. V. McGann. A fractional order recovery SIR model from a stochastic process, \textit{Physica A: Statistical Mechanics and its Applications}, vol. 452, pp. 86-93, 2016.

\bibitem{angstmann2017} C. N. Angstmann, A. M. Erickson, B. I. Henry, A. V. McGann, J. M. Murray, J. A. Nichols. Fractional order compartment models, \textit{SIAM Journal on Applied Mathematics}, vol. 77, no. 2, pp. 430–446, 2017.

\bibitem{baz2} E. G. Bazhlekova. Fractional evolution equations in Banach spaces, PhD Thesis, Technische Universiteit Eindhoven, 2001.

\bibitem{car} J. M. Carcione, F. J. Sanchez-Sesma, F. Luzón and J. J. P. Gavilán. Theory and simulation of time-fractional fluid diffusion in porous media, \textit{Journal of Physics A: Mathematical and Theoretical}, vol. 46, no. 34, Article ID 345501, 2013. %%%

\bibitem{CaFe1} P. M. Carvalho-Neto and R. Fehlberg Júnior. The Riemann-Liouville fractional integral in Bochner-Lebesgue spaces I, \textit{Communications on Pure and Applied Analysis}, vol. 21, no. 11, pp. 3667-3700, 2022.

\bibitem{chechkin} V. Chechkin, R. Gorenflo and I. M. Sokolov. Fractional diffusion in inhomogeneous media, \textit{Journal of Physics A: Mathematical and General}, vol. 38, no. 42, pp. 679-684, 2005. %%%

\bibitem{CoDe1} C. G. Colcord and W. E. Deming. The Minimum in the Gamma Function, \textit{Nature}, vol. 135, p. 917, 1935.

\bibitem{compte} A. Compte. Stochastic foundations of fractional dynamics, \textit{Physical Review E}, vol. 53, no. 4, pp. 4191-4193, 1996. %%%

\bibitem{fedotov} S. Fedotov and S. Falconer. Subdiffusive master equation with space-dependent anomalous exponent and structural instability, \textit{Physical Review E}, vol. 85, no. 3, Article ID 031132, 2012. %%%

\bibitem{gerritse} M. Gerritse. COVID-19 transmission in cities, \textit{European Economic Review}, vol. 150, Article ID 104262, 2022. %%%

\bibitem{gokima} R. Gorenflo, A. A. Kilbas, F. Mainardi, and S. V. Rogosin. \textit{Mittag-Leffler Functions, Related Topics and Applications}, Springer, Berlin, 2014. %%%

\bibitem{goulart} A. Goulart, M. Lazo and J. Suarez. A new parameterization for the concentration flux using the fractional calculus to model the dispersion of contaminants in the planetary boundary layer, \textit{Physica A: Statistical Mechanics and its Applications}, vol. 518, pp. 38-49, 2019. %%%

\bibitem{GuQi1} B.-N. Guo and F. Qi. Refinements of Lower Bounds for Polygamma Functions, \textit{Proceedings of the American Mathematical Society}, vol. 141, no. 3, pp. 1007–1015, 2013.

\bibitem{HaLiPo1} G. H. Hardy, J. E. Littlewood and G. Pólya. \textit{Inequalities}, 2nd ed., Cambridge University Press, Cambridge, 1952. %%%

\bibitem{hilfer} R. Hilfer and L. Anton. Fractional master equations and fractal time random walks, \textit{Physical Review E}, vol. 51, no. 2, pp. R848-R851, 1995. %%%

\bibitem{kian} Y. Kian, E. Soccorsi and M. Yamamoto. On time-fractional diffusion equations with space-dependent variable order, \textit{Annales Henri Poincaré}, vol. 19, pp. 3855–3881, 2018. %%%

\bibitem{liu} Q. Liu and L. Chen. Time-space fractional model for complex cylindrical ion-acoustic waves in ultrarelativistic plasmas, \textit{Complexity}, vol. 2020, Article ID 417942, 2020. %%%

\bibitem{lu2022} Z. Lu, Y. Yu, G. Ren, C. Xu, X. Meng. Global dynamics for a class of reaction–diffusion multigroup SIR epidemic models with time fractional-order derivatives, \textit{Nonlinear Analysis: Modelling and Control}, vol. 27, no. 1, pp. 142–162, 2022.

\bibitem{makris} N. Makris, G. Dargush and C. M. Constantinou. Dynamic analysis of generalized viscoelastic materials, \textit{Journal of Engineering Mechanics}, vol. 119, no. 8, pp. 1663-1679, 1993. %%%

\bibitem{metzler} R. Metzler and J. Klafter. The random walk’s guide to anomalous diffusion: a fractional dynamics approach, \textit{Physics Reports}, vol. 339, no. 1, pp. 1-77, 2000. %%%

\bibitem{Mi1} J. Mikusiński. \textit{The Bochner Integral}, Birkhäuser, Basel, 1978.

\bibitem{sun} H. G. Sun, W. Chen and Y. Q. Chen. Variable-order fractional differential operators in anomalous diffusion modeling, \textit{Physica A}, vol. 388, pp. 4586-4592, Elsevier, 2009.

\bibitem{west} B. J. West, R. Grigolini, R. Metzler and T. F. Nonnenmacher. Fractional diffusion and Lévy stable processes, \textit{Physical Review E}, vol. 55, 1997.

\bibitem{yin} H. Yin, T. Sun, L. Yao, Y. Jiao, L. Ma, L. Lin, J. C. Graff, L. Aleya, A. Postlethwaite, W. Gu, and H. Chen. Association between population density and infection rate suggests the importance of social distancing and travel restriction in reducing the COVID-19 pandemic, \textit{Environmental Science and Pollution Research International}, vol. 28, no. 30, pp. 40424–40430, 2021.

\end{thebibliography}
\end{document}